\documentclass[12pt]{amsart}
\usepackage[utf8]{inputenc}
\usepackage[dvipsnames]{xcolor}
\usepackage{fontenc,fancybox}
\usepackage{enumitem} 
\usepackage{tikzsymbols}
\usepackage{amsthm,amsmath,amsfonts,graphics,graphicx,geometry}
\usepackage{cite, enumitem, color, xargs, xifthen, fancyhdr, subcaption, wrapfig}
\usepackage{tikz, tkz-tab, pgfplots, float, caption, listings, mdframed}
\pgfplotsset{compat=1.18}
\usepackage{amssymb, amsmath, esvect, mathtools, siunitx}

\usepackage{stmaryrd}
\usepackage{dsfont} 
\usepackage[colorlinks=true, urlcolor=blue,linkcolor=blue, citecolor=blue]{hyperref}

\makeatletter

\@namedef{subjclassname@2010}{
  \textup{2010} Mathematics Subject Classification}

\makeatother
\newtheorem{thm}{Theorem}[section]

\newtheorem*{AssL}{Assumption (Law)}
\newtheorem*{AssT}{Assumption (Tightness)}

\newtheorem{cor}[thm]{Corollary}
\newtheorem{lem}[thm]{Lemma}
\newtheorem{pro}[thm]{Proposition}
\theoremstyle{definition}

\newtheorem{rem}{Remark}

\numberwithin{equation}{section}

\numberwithin{equation}{section}

\def\A{\mathcal{A}}

\def\F{\mathbb{F}}

\def\E{\mathbb{E}}

\def\N{\mathbb{N}}
\def\C{\mathbb{C}}
\def\Ecal{\mathcal{E}}
\def\Fcal{\mathcal{F}}
\def\Ccal{\mathcal{C}}
\def\Ncal{\mathcal{N}}

\def\Scal{\mathcal{S}}
\def\R{\mathbb{R}}
\def\P{\mathbb{P}}

\def\Z{\mathbb{Z}}
\def\Bi{\mathrm{Bi}}
\def\AS{\mathrm{AS}}
\def\Kl{\mathrm{Kl}}
\def\mod{\mathrm{mod}}
\def\B{\mathrm{S}}
\def\M{\mathcal{M}}
\def\Mcal{\mathcal{M}}
\def\Mp{C_\varphi}

\def\S{\mathrm{S}}

\def\X{\mathbb{X}}

\def\un{\mathds{1}}

\def\E{\mathbb{E}}

\title{On the distribution of the maximum of partial sums}

\newcommand{\newabstract}[1]{%
  \par\bigskip
  \csname otherlanguage*\endcsname{#1}%
  \csname captions#1\endcsname
  \item[\hskip\labelsep\scshape\abstractname.]
}

\title[The maximum of partial sums in families of trace functions]{The distribution and the structure of the maximum of partial sums in families of trace functions}

\author{Kilian Lebreton}

\address{Universit\'e de Lorraine, CNRS, IECL, 
F-54000 Nancy, France}

\email{kilian.lebreton@univ-lorraine.fr}

\begin{document}

\baselineskip=17pt


\begin{abstract} 
In \cite{ABL21}, Autissier, Bonolis and Lamzouri obtained uniform estimates for the distribution function of the maximum of partial sums of a class of families of $m$-periodic complex-valued functions under certain conditions. For example, these conditions are verified by Kloosterman or Birch sums.

In this article, under the same assumptions, we obtain an improved estimate for the tail of the distribution of the maximum of partial sums in these families, which gives for the first time an asymptotic formula for the logarithm of the distribution function, in a large uniform range. 

Furthermore, we prove a "structure theorem" for the maximum of partial sums of a family of complex $m$-periodic functions in our class, which shows that universally over this class, most of the partial sums with large norm are almost purely imaginary, and the maximum is attained around $m/2$. This is in sharp contrast with the results of Bober-Goldmakher-Granville-Koukoulopoulos \cite{BGGK18}, Lamzouri \cite{Lam24} and Lamzouri-Nath \cite{LN24} for the distribution of the maximum of character sums in various families of Dirichlet characters.
\end{abstract}


\maketitle

\section{Introduction}

\subsection{Introduction}
Let $m\geq 2$ be an integer, and $\varphi:\Z/m\Z\to \C$ a complex-valued function which we extend to an $m$-periodic function $\varphi:\Z\to \C$.   
Understanding the true size of the maximum of partial sums 
\[\M(\varphi) := 
\max_{t\in[0,1]}\left|
\frac{1}{\sqrt{m}}\sum_{0\leq n\leq mt} \varphi(n)
\right|,
\] is a notoriously difficult problem.
To study this question, we introduce the normalized discrete Fourier transform of $\varphi$.
For $h\in\Z$, it is defined by
\[
\widehat{\varphi}(h):= \frac{1}{\sqrt{m}}\sum_{n\in \Z/m\Z} \varphi(n)e_m(-hn),
\]
where $e_m(n) = e^{\frac{2\pi i}{m}n}$. The best general estimate for $\M(\varphi)$ remains the classical inequality of Pólya and Vinogradov (1918), which states that
\[
\M(\varphi) \ll ||\widehat{\varphi}||_\infty \log m .
\] 
In many situations arising in analytic number theory, one has bounds of the form
\[
|\widehat{\varphi}(h)|\leq \Mp \qquad \text{for all } h\in\Z,
\]
which typically follow from deep results in algebraic geometry such as the Weil bounds or Deligne’s equidistribution theorem.

In the case where $\varphi =\chi$ is a non-principal Dirichlet character modulo $m$, Montgomery and Vaughan \cite{MV77} obtained the stronger bound
\[
\M(\chi )\ll \log\log m,\] assuming the Generalized Riemann Hypothesis (GRH).
This bound is optimal by a result of Paley \cite{P32}, who obtained a corresponding lower bound for an infinite family of primitive quadratic characters.
More recently, for a primitive character $\chi$ modulo $m$ of fixed odd order $g$, the works \cite{GS07}, \cite{Gol12}, \cite{GL12}, \cite{LM22}, and \cite{Man26} led to the following improved bounds
\[
\M(\chi )
\ll 
\left\{
\begin{array}{l cc l}
         (\log m)^{\beta_g} (\log_2 m)^{-1/4+o(1)}& \text{unconditionally},\\
     (\log_2 m)^{\beta_g}
(\log_3 m)^{-1/4}
(\log_5 m)^{1-\beta_g} & \text{assuming GRH}. 
\end{array}
\right.
\]
where $\beta_g := \dfrac{g}{\pi}\sin\left(\dfrac{\pi}{g}\right)$ and
$\log_{k+1} m := \log(\log_k m)$ for $k\geq 1$.
Moreover, Mangerel \cite{Man26} unconditionally obtained a matching lower bound of the above conditional bound, proving the correct order of magnitude under GRH.

Let $p\geq 3$ be a prime and let $\mathbb{F}_p$ denote the finite field of $p$ elements. Livné \cite{Liv87} proved that the complete normalized Birch sums
\[
\Bi_p(a):= \frac{1}{\sqrt{p}}\sum_{n\in \F_p} e_p(n^3+an),
\]as $a$ varies over $\F_p$ with the uniform probability measure, become equidistributed in $[-2,2]$ with respect to the Sato--Tate measure, 
\[
    \frac{1}{\pi}\sqrt{1-\frac{x^2}{4}}dx,
    \]
as $p$ tends to $\infty$. 
In \cite{KS16}, Kowalski and Sawin investigated the following partial Birch sums. 
For $t \in [0,1]$, they define
\[
\Bi_p(a, t):= \frac{1}{\sqrt{p}}\sum_{0\leq n\leq pt} e_p(n^3+an).
\]
They also consider the associated continuous polygonal path 
$\big(\widetilde{\Bi}_p(a,t)\big)_{t\in[0,1]}$, obtained by linear interpolation 
between consecutive partial sums. More precisely, setting  
$z_k=\Bi_p(a,k/p)$ for $0\leq k\leq p$, the segment $[z_k,z_{k+1}]$ 
is parametrized linearly for $t\in[k/p,(k+1)/p]$, that is
\[
\widetilde{\Bi}_p(a,t)
= \frac{1}{\sqrt{p}}\sum_{0\leq n\leq \lfloor pt \rfloor}\varphi_a(n)
+ \frac{1}{\sqrt{p}}\big(pt-\lfloor pt\rfloor\big)\varphi_a(\lfloor pt\rfloor+1).
\]
Kowalski and Sawin
proved that as $a$ varies in $\F_p$, the process $(\overset{\sim}{\Bi_p}(a,t))_{t\in[0,1]}$ converges in law in $\Ccal([0,1])$ as $p$ tends to $\infty$ to the process $(\S_\X(t))_{t\in [0,1]}$, defined for $t\in[0,1]$ by
\begin{equation}\label{def S_X(t)}
\S_\X(t) := t \X(0)+ \sum_{h\in\Z^*} \frac{e^{2\pi i ht}-1}{2\pi i h} \X(h),
\end{equation}
where $(\X(h))_{h\in\Z}$ is a sequence of independent random variables with the Sato--Tate measure on $[-2,2]$.
This means that for any bounded map $F : \Ccal([0,1]) \to \C$ continuous with respect to the topology of the uniform convergence, we have 
\[
\underset{p \to \infty }{\lim}\E[F(\overset{\sim}{\Bi_p}(a,\cdot) )] = \E [ F(\S_\X) ].
\] 
Furthermore, by taking $F = ||\cdot||_\infty$, Kowalski and Sawin showed in \cite{KS16} that
there exists a positive constant $c$ such that we have for all $V>0$,
\begin{align*}
c^{-1}\exp(-\exp(cV))\leq
& \underset{p\rightarrow \infty}{\lim} \frac{1}{p}\left| \Big\{
a\in \F_p 
\:: \
||\Bi_p(a,\cdot) ||_\infty \geq V
\Big\}\right| 
\leq c\,\exp(-\exp(V/c)).
\end{align*}
This suggests that the maximum of partial sums can be as large as $c\log\log p$. In \cite{Lam20}, Lamzouri improved this result, showing that for all real numbers $1\leq V\leq \frac{2}{\pi}\log_2 p-2\log_3 p$, we have:
\begin{align}\label{Lamzouri estimate}
\exp\left(-A_0\exp\left(\frac{\pi}{2}V\right)\left( 1+O\left(\sqrt{V}e^{-\pi V/4}\right)\right)\right)&\leq 
\frac{1}{p}\left| \Big\{
a\in \F_p 
\:: \
||\Bi_p(a,\cdot) ||_\infty \geq V
\Big\}\right| \\
\label{Lamzouri estimate2}
&\leq \exp\left(-C\exp\left(\left(\frac{\pi}{2}-\delta\right)V\right)\right).
\end{align}
where
$
\delta = \frac{4\pi-\pi^2}{2\pi+8},
$
\(C\) is a positive constant and
$
A_0:=\exp\left(
-\gamma-1-\frac12\int_0^\infty \frac{f_{\X}(u)}{u^2}\,du
\right),
$
with \(f_{\X}\) defined in \eqref{def f_X} for a Sato--Tate random variable \(\X\), and \(\gamma\) denoting the Euler--Mascheroni constant.
An immediate consequence of the lower bound \eqref{Lamzouri estimate} is the existence of $a\in\F_p$ such that
\[\left|\frac{1}{\sqrt{p}} \sum_{0\leq n\leq p/2} e_p(n^3+an)\right| \geq \left(\frac{2}{\pi}+o(1)\right)\log\log p.
 \]

We should note that Bonolis \cite{Bon22} obtained this result with a smaller constant, using a different method. In \cite{ABL21}, Autissier, Bonolis and Lamzouri improved the upper bound \eqref{Lamzouri estimate2} for the distribution function, showing that there exists a constant $V_0$ such that for all real numbers $V_0 \leq V \leq \frac{2}{\pi}(\log_2 p- 2\log_3 p) -V_0$ we have
\[\frac{1}{p}\left| \Big\{
a\in \F_p 
\:: \
||\Bi_p(a,\cdot) ||_\infty \geq V
\Big\}\right| =
\exp\left(-\exp\left(\frac{\pi }{2} V
+O(1)\right) \right). 
\]
As a corollary of our general Theorem \ref{THM Ncal_m(V)} below, we improve this result, showing that the lower bound \eqref{Lamzouri estimate} gives the correct order of magnitude of this distribution function. More precisely, for the same range of values of $V$, one has
\[\frac{1}{p}\left| \Big\{
a\in \F_p 
\:: \
||\Bi_p(a,\cdot) ||_\infty \geq V
\Big\}\right| = \exp\left(-A_0\exp\left(\frac{\pi}{2}V\right)\left( 1+O\left(\frac{1}{V}\right)\right)\right). 
\]

\subsection{Main results}

Let $\{\varphi_a\}_{a\in\Omega_m}$ be a family of $m-$periodic 
complex-valued functions. We define the partial sums of these $\varphi_a$, for all $a\in\Omega_m$, by
\begin{equation}\label{def S_m(a,t)}
\B_m(a,t):= \frac{1}{\sqrt{m}}\sum_{0\leq n \leq mt} \varphi_a(n).
\end{equation}
By the discrete Plancherel formula, the partial sums can be expressed in terms of complete sums, as follows:
\begin{equation}\label{S_m(a,t) fourier transform}
    \B_m(a,t)= \sum_{h\in\Z/m\Z}
    \gamma_m(h,t)\widehat{\varphi_a}(h), \qquad \text{where}\;\;\gamma_m(h,t) 
:= \frac{1}{m}\sum_{0\leq k \leq mt}  e_m(hk).
\end{equation}

We now assume that for all $h \in \Z $ the distribution of $\widehat{\varphi_a}(h)$ converges to that of a bounded real random variable $\X(h)$, as $a$ varies over $ \Omega_m$ and $ m$ tends to $\infty$. More precisely, we assume the following conditions, corresponding to Assumptions $2$ and $3$ with $\eta = 1/2$ in \cite{ABL21}. The only difference is the inclusion of the index $h=0$. We describe several situations where these conditions are satisfied.
\begin{AssL}\label{Ass Law}
    \hypertarget{Ass Law}{The convergence in law of the Fourier transform:}\\ There exist $\Mp>0$ and a random variable $\X$ on $[-\Mp,\Mp]$ such that
\begin{itemize}
    \item $\X$ is symmetric, that is, $\X \sim -\X$, where $\sim$ denotes equality in distribution.

    \item For all $\varepsilon >0$, we have $\P(\X>\Mp-\varepsilon)\gg \varepsilon^{C_2},$ for some positive constant $C_2$.

    \item For all $a\in\Omega_m$ and $h\in \Z /m\Z,$ we have $\widehat{\varphi_a}(h)\in [-\Mp,\Mp].$

    \item Let $\{\X(h)\}_{h\in\Z}$ be a sequence of independent and identically distributed random variables with the same distribution as $\X$. There exists an absolute constant $C_1>1$, such that for all positive integers $k\leq \log m/\log\log m,$ and all $(h_1, ...,h_k)\in (-m/2, m/2]^k,$ we have
\begin{equation}\label{eq moment Ass Law}
\frac{1}{|\Omega_m|}\sum_{a\in\Omega_m} \widehat{\varphi_a}(h_1)\cdots \widehat{\varphi_a}(h_k)
 =
 \E[\X(h_1)\cdots \X(h_k)] +O\left(\frac{C_1^k}{m^{1/2}}\right).
\end{equation}

\end{itemize}
\end{AssL}

Assumption (\hyperlink{Ass Law}{Law}) appears naturally
for many families of $\ell-$adic trace functions, thanks to several ingredients from algebraic geometry, notably Deligne's equidistribution theorem, Katz's work on monodromy groups and the Goursat-Kolchin-Ribet criterion.  
Under Assumption (\hyperlink{Ass Law}{Law}), we can already obtain information on the distribution of the maximum of the partial sums at a fixed $t_0 \in [0,1]$, namely
\begin{equation}\label{def Nm(t_0, V)}
   N_m(t_0, V) := \frac{1}{|\Omega_m|}\left|\left\{
a\in \Omega_m\;\; \big|\;\; \left|
\frac{1}{\sqrt{m}}\sum_{0\leq n\leq t_0 m} \varphi_a(n)
\right|> V
\right\}\right|. 
\end{equation}
For instance, by assuming Assumption (\hyperlink{Ass Law}{Law}),  it follows from \cite[Theorem 7.1]{ABL21} that there exists a constant $V_0$ such that 
for all \(V\) satisfying $1\leq V \leq$
$\frac{\Mp}{\pi} \left(\log_2 m-2\log_3 m\right)-V_0$, we have
\begin{equation}\label{eq ABL Nm(1/2,V)}
N_m(1/2,V)
=
\exp\left[- \frac{\Mp}{2 e} \exp\left(-\gamma-I_\X \right)\exp\left(\frac{\pi }{\Mp}V\right)
\left(1+O\left(\frac{1}{V}\right)\right) 
\right].
\end{equation}
where $
I_\X $ is defined by
\begin{equation}\label{def f_X}
I_\X :=\frac{1}{\Mp}\int_0^\infty \frac{f_{\X}(u)}{u^2}du, \qquad
f_\X(u) :=\left\{
\begin{array}{l cc l}
         \log \E[e^{u\X}] & \text{if}  & |u|<1,\\
     \log \E[e^{u\X}]- \Mp|u| & \text{if}  & |u|\geq 1. 
\end{array}
\right.
\end{equation}
In particular, this yields the following lower bound for the distribution function of the maximum $\M(\varphi_a)$, as $m$ tends to $\infty$,
\[\Ncal_m(V)\geq \exp\left[- \frac{\Mp}{2 e} \exp\left(-\gamma-I_\X \right)\exp\left(\frac{\pi }{\Mp}V\right)
\left(1+O\left(\frac{1}{V}\right)\right) 
\right],\] 
where
\[
\Ncal_m(V):=
\frac{1}{|\Omega_m|}\left|\left\{
a\in \Omega_m\;\; \big|\;\;\M(\varphi_a)> V
\right\}\right|.
\]
However, to obtain a matching upper bound, we need another assumption.
\begin{AssT}\label{Ass Tightness} \hypertarget{Ass Tightness}{
There exists $\alpha\geq 1$ and $0<\delta<1/4$ such that:}
\[
\frac{1}{|\Omega_m|}\sum_{a\in\Omega_m}
\max_{|I|\leq m^{1/2+\delta}}\left|
\frac{1}{\sqrt{m}}\sum_{n\in I} \varphi_a(n)
\right|^{\alpha } \ll
m^{-\delta}.
\]
\end{AssT}
When the cardinality of the interval $I$ is close to $\sqrt{m}$, the trivial bound is no longer effective, and the $O(1/\sqrt{m})$ error term in the probabilistic approximation is not sufficient to conclude. 
Proposition \ref{Pro Assumptions implication} shows that each of the three sets of assumptions considered in \cite{ABL21} implies Assumption (\hyperlink{Ass Tightness}{Tightness}).
As a consequence, the results of Autissier, Bonolis and Lamzouri \cite{ABL21}, together
with the translation-invariance of the families considered (see Remark \ref{rem hat varphi_a(h) = hat varphi a-h(0)}), show
that Assumptions (\hyperlink{Ass Law}{Law}) and (\hyperlink{Ass Tightness}{Tightness}) are satisfied by many families of trace
functions when $m = p$ is a large prime:
\begin{itemize}
    \item The classical Kloosterman sums, for $(a,b)\in (\F_p^*)^2 = \Omega_p $ and $n\in \F_p^*,\, \varphi_{a,b}(n) = e_p(an+b\overline{n})$ where $\overline{n}$ denotes the inverse of $n $ modulo $p$, and by convention $\varphi_{a,b}(0)=0$. Thus, we have
    \[
    \Kl_{p}(a,b) = \frac{1}{\sqrt{p}}\sum_{n\in \F_p^*}
    e_p(an+b\overline{n}).
    \] For this family, the limiting distribution
    $\X$ is a random variable with distribution equal to the Sato--Tate measure on $[-2,2]$.
Fixing \(b=1\) yields the corresponding one-parameter family, which satisfies Assumption~(\hyperlink{Ass Law}{Law}), whereas Assumption~(\hyperlink{Ass Tightness}{Tightness}) remains open in this case; see Remark 3.4 in \cite{KS16}.
    
    \item The Birch sums, for \(a\in\F_p=\Omega_p\), let
\(\varphi_a(n)=e_p(n^3+an)\). Then
    \[
    \Bi_p(a) = \frac{1}{\sqrt{p}}\sum_{n\in \F_p}  e_p(n^3 + an).
    \]
    where the limiting distribution $\X$ is a random variable with the Sato--Tate distribution on $[-2,2]$.

    \item The Artin-Schreier sums. Let $g\in\Z[T]$ be an odd polynomial of degree $2r+1\geq 3$. For any prime $p> 2r+1$ and for any $a\in  \F_p^*  = \Omega_p$, we let $\varphi_a(n) = e_p(g(n) + an)$ and
    \[
    \AS_p(a) = \frac{1}{\sqrt{p}}
    \sum_{n\in \F_p}  e_p(g(n) + an).
    \]

    \item The twisted $r$-th hyper-Kloosterman sums, for $r\geq 2$ and $(a,b) \in (\F_p^*)^2 = \Omega_p,$ 
    \[
    \varphi_{a,b}(n) := \frac{1}{p^{(r-1)/2}}
    \;\left(
    \sum_{x_1,...,x_{r-1}\in\F_p^*}
    e_p(x_1+...+x_{r-1}+ 
    \overline{an x_1\cdots x_{r-1}})\right)e_p(bn).
    \]
    and
    \[
    \Kl_{p,r}(a,b) = \frac{1}{p^{r/2}}\sum_{n = 1}^{p-1}
    \;\left(
    \sum_{x_1,...,x_{r-1}\in\F_p^*}
    e_p(x_1+...+x_{r-1}+ 
    \overline{an x_1\cdots x_{r-1}})\right)e_p(bn).
    \]

    \item In general, the true order of magnitude of the maximum of partial sums in these families remains unknown. However, in \cite{ABL21}, the authors construct a family $\{\varphi_a\}_{a\in\Omega_m}$ of $m$-periodic complex-valued functions satisfying Assumptions (\hyperlink{Ass Law}{Law}) and (\hyperlink{Ass Tightness}{Tightness}) for which
\[
\frac{1}{\sqrt{m}}\sum_{0\leq n\leq m/2}\varphi_1(n)
=
\frac{1}{\pi i}\log m + O(1).
\]
Thus, some normalized partial sums of
\(\{\varphi_a\}_{a\in\Omega_m}\) may be as large as \(\log m\), illustrating the difficulty of improving the Pólya–Vinogradov inequality in complete generality.

\item We may also add the following example, which does not appear in \cite{ABL21}: the family of quadratic Gauss sums associated with $\chi=\left(\frac{\cdot}{p}\right)$. For all $a\in \Omega_p= \F_p^*$, we define
    \[
    \varphi_a(n) =\frac{\sqrt{p}}{\tau(\chi)} \left(\frac{n}{p}\right)e_p(an),
    \]
    where $\tau(\chi) = \sum_{n=1}^{p-1} \left(\frac{n}{p}\right)e_p(n)$.
\end{itemize}

In \cite[Theorem 1.2]{ABL21}, the authors prove that if  $\{ \varphi_a\}_{a\in\Omega_m}$ satisfies certain assumptions that are roughly equivalent to Assumptions (\hyperlink{Ass Law}{Law}) and (\hyperlink{Ass Tightness}{Tightness}), there exists a constant $V_0$ such that for all \(V\) satisfying $V_0 \leq V \leq (\Mp/\pi )(\log_2 m- 2\log_3 m) -V_0$ we have
\[\Ncal_m(V)=
\exp\left(-\exp\left(\frac{\pi }{\Mp}V
+O(1)\right) \right). 
\]
In this article, we improve this result by obtaining a very accurate estimate for $\Ncal_m(V)$, which in particular leads to an asymptotic formula for $\log \Ncal_m(V)$ in a similar range of $V$.

\begin{thm}\label{THM Ncal_m(V)} Let $m$ be large enough and $\{\varphi_a\}_{a\in\Omega_m}$ be a family of $m-$periodic complex-valued functions satisfying Assumptions (\hyperlink{Ass Law}{Law}) and (\hyperlink{Ass Tightness}{Tightness}) above. There exists a constant $ V_0>0$ depending only on $\X$ such that for all \(V\) satisfying
\[ V_0\leq V \leq \frac{\Mp}{\pi} \left(
\log\log m-2\log\log\log m\right)-V_0,\] 
we have 
\begin{align*}
\Ncal_m(V)
&=
\exp\left[- \frac{\Mp}{2 e} \exp\left(-\gamma-I_\X \right)\exp\left(\frac{\pi }{\Mp}V\right)
\left(1+O\left(\frac{1}{V}\right)\right) \right],
\end{align*}
where $\gamma$ is the Euler-Mascheroni constant and  $
I_\X $ is defined in \eqref{def f_X}.

\end{thm}

\begin{rem}
Let $C_0 := \lim_{m\to \infty} \frac{C_\varphi}{\pi}\left(\log \left(\frac{2 \log |\Omega_m| }{\Mp \log m}\right)
+1+\gamma+I_\X\right)$.
If the range of validity of $V$ in Theorem \ref{THM Ncal_m(V)} extends to $ (\Mp/\pi)\log\log m+C_0+\varepsilon$, for $\varepsilon>0$, we would obtain
\begin{align*}\sup_{a\in \Omega_m}\Mcal(\varphi_a)&
= \frac{C_\varphi}{\pi}\log\log m +\frac{C_\varphi}{\pi}\left(\log 2- \log  \Mp 
+1+\gamma+I_\X\right) \\&
+\frac{C_\varphi}{\pi}\log \left(\frac{\log |\Omega_m| }{\log m}\right)
+O\left(\frac{1}{\log\log m}\right).
\end{align*}  

\end{rem}

For each \(a\in \Omega_m\), we define \(t_a \in [0,1]\) and \(\theta_a \in [0,1]\) such that
\[\M(\varphi_a) = 
\max_{t\in[0,1]}\max_{\theta\in[0,1]}\left|\Re\left(
\frac{e^{\pi i\theta}}{\sqrt{m}}\sum_{0\leq n\leq mt} \varphi_a(n)\right)
\right| = \left|\Re\left(
\frac{e^{\pi i\theta_a}}{\sqrt{m}}\sum_{0\leq n\leq m t_a} \varphi_a(n)\right)\right|.
\]

In the regime where the maximum is large, the following structure theorem describes precisely the location of \(t_a\) and \(\theta_a\) for the vast majority of elements \(a\in \Omega_m\).

\begin{thm}\label{thm structure simp} Assume Assumptions (\hyperlink{Ass Law}{Law}) and (\hyperlink{Ass Tightness}{Tightness}). There exists a constant $ V_0>0$ depending only on $\X$ such that for all \(V\) satisfying $V_0\leq V \leq$
$ \frac{\Mp}{\pi}(\log_2 m-2\log_3 m)-V_0$, there exists a subset $\A_m(V)\subset \left\{
a\in \Omega_m\: \big|\:\M(\varphi_a)> V
\right\}$ and a constant $C>0$ such that:
\[
\left|\A_m(V) \right|= \left|\left\{
a\in \Omega_m\;\big|\;\M(\varphi_a)> V
\right\}\right|
\left[1 +O\left(\exp\left(- 
\frac{Ce^{\pi V/\Mp}}{V^{4/5}}
\right)\right)\right],
\]
and such that for all $ a \in \A_m(V),$ we have
\begin{enumerate}
    \item $\M(\varphi_a)=V+ O\left(V^{-4/5}\right)$,
    \item $t_a = 1/2+ O\left(V^{-1/5}\right)$,
    \item $\theta_a = 1/2+ O\left(V^{-4/5}\right)$.
\end{enumerate}

\end{thm}
 
This structure theorem is closely related to \cite[Theorem 1.2]{BGGK18} and \cite[Theorem 1.3]{LN24}, which study the location of the point $t_\chi$ at which the maximum of $\left| \sum_{1\leq n\leq q t } \chi(n) \right|$ is attained for a given non-principal
Dirichlet character $\chi (\mod \,  q)$. In \cite[Theorem 1.2]{BGGK18}, the authors show that, \(t_\chi\) is close to $1/2$ for the vast majority of non-principal Dirichlet characters $\chi (\mod \,q)$, where $q$ is an odd prime. Moreover, \cite[Theorem 2.3]{BGGK18} shows that for the vast majority of non-principal Dirichlet characters $\chi (\mod \,q)$ satisfying $\chi(-1) = 1$, $t_\chi $ is close to $ 1/3$ or $2/3$.
In a different direction, \cite[Theorem 1.3]{LN24} shows that, for most primitive cubic characters \(\chi \,(\mathrm{mod}\, q)\) with \((q,3)=1\), the point \(t_\chi\) is close to a rational number with a large denominator.

\section{Estimate of the Laplace transform for the probabilistic model}
\subsection{Probabilistic model}
Let \(t\in[0,1]\). Endowing \(\Omega_m\) with the uniform probability measure, Assumption~(\hyperlink{Ass Law}{Law}) implies that the random variable \(\B_m(a,t)\) converges in law to
\begin{equation}\label{eq def gamma(h,t)}
\S_\X(t) 
:= \sum_{h\in\Z} \gamma(h,t)\X(h),
\qquad\text{where}\qquad
\gamma(h,t):=
\begin{cases}
t & \text{if } h=0,\\[0.2cm]
\dfrac{e^{2\pi i ht}-1}{2\pi i h} & \text{if } h\neq 0,
\end{cases}
\end{equation}
and where \(\{\X(h)\}_{h\in\Z}\) is a sequence of independent and identically distributed random variables with the same distribution as \(\X\) in Assumption~(\hyperlink{Ass Law}{Law}).

We study the Laplace transform of the probabilistic model \(\S_\X(t)\). Since \(\S_\X(t)\) is complex-valued, we study instead the real-valued random variables
$
\Re\!\left(e^{\pi i\theta}\S_\X(t)\right)\!,
$  for $(t,\theta)\in[0,1]^2.
$
Since \(\X\) is symmetric, we define
\[
\gamma(h,t,\theta)
:= \left|\Re\!\left(e^{\pi i\theta}\gamma(h,t)\right)\right|,
\qquad (t,\theta)\in[0,1]^2,\ h\in\Z.
\]
Then
\begin{equation*}
\Re\!\left(e^{\pi i\theta}\S_\X(t)\right)
\sim
\sum_{h\in\Z}\gamma(h,t,\theta)\X(h)
=
t|\cos(\pi\theta)|\,\X(0)
+
\sum_{h\in\Z^*}
\frac{|\cos(\pi\theta+\pi ht)\sin(\pi ht)|}{\pi|h|}
\,\X(h).
\end{equation*}
Indeed, for \((t,\theta)\in[0,1]^2\), $
\gamma(0,t,\theta)
=
t|\cos(\pi\theta)|,
$
while for \(h\in\Z^*\),
\begin{align*}
\gamma(h,t,\theta) =
\left|
\Re\!\left(
\frac{
e^{\pi i\theta+\pi i ht}
\left(e^{\pi i ht}-e^{-\pi i ht}\right)
}{2\pi i h}
\right)
\right|
&=
\frac{|\cos(\pi\theta+\pi ht)\sin(\pi ht)|}{\pi|h|}.
\end{align*}

\subsection{Optimal case $t=1/2$ and $\theta = 1/2$}
In this section, we find a uniform upper bound for the Laplace transform of the probabilistic model. Indeed, we will show that for all $t,\theta\in [0,1]$,
\[
\log \E \left[\exp \left(s \cdot \sum_{h\in \Z} \gamma(h,t,\theta)\X(h)\right) \right] 
\leq 
\frac{\Mp}{\pi}s\log s +B_{t,\theta}s +O\left(\frac{s}{\log s}\right),
\]
where $B_{t,\theta}$ is a constant depending on $t$ and $\theta$.
Here, the uniformity of the error term is crucial. However, we begin with the imaginary part of the half sums, corresponding to \(t=1/2\) and \(\theta=1/2\). In this case, we have
\begin{equation}\label{eq gamma(h,1/2,1/2)}
    \forall k \in \Z, \qquad
\gamma\left(2k,\frac{1}{2},\frac{1}{2}\right)=0,
\quad
\gamma\left(2k+1,\frac{1}{2},\frac{1}{2}\right)
=
\frac{1}{|2k+1|\pi}.
\end{equation}
Recently, Autissier, Bonolis, and Lamzouri proved in \cite{ABL21} the following asymptotic formula for the Laplace transform of $
\Im\!\left(\S_\X(1/2)\right),$ which in particular yields \eqref{eq ABL Nm(1/2,V)}.

\begin{pro}\label{pro kappa ABL21}\cite[Proposition 5.2]{ABL21} For all $s\geq 2$, we have
\begin{equation}\label{eq asymp 1/2 1/2}
\log \E \left[\exp \left(s \cdot \sum_{h\in \Z} \gamma\left(h,\frac{1}{2},\frac{1}{2}\right)\X(h)\right)\right] 
= \frac{\Mp}{\pi}s\log s+
\kappa s +O\left(\log^2 s \right),
\end{equation}
where
\begin{equation}\label{def kappa}
\kappa :=\frac{\Mp}{\pi}
\left(\log\left(\frac{2}{\pi}\right)+\gamma+ I_\X \right).
\end{equation}
where $\gamma$ is the Euler-Mascheroni constant and  $
I_\X $ is defined in \eqref{def f_X}.

\end{pro}
The constants appearing above are optimal for all $t,\theta\in [0,1]$. Indeed, we obtain the following theorem in Section \ref{section 2.3 preuve thm 2.2}.
\begin{thm}\label{THM asymp t theta} There exists $C_3>0$ such that uniformly for $t, \theta$ and for $s\geq \pi$, we have
\begin{align*}
\log \E \left[\exp \Big(s \cdot \sum_{h\in \Z}
     \gamma(h,t,\theta)\X(h)\Big) \right] 
\leq & \log \E \left[\exp \left(s \cdot \sum_{h\in \Z} \gamma\left(h,\frac{1}{2},\frac{1}{2}\right)\X(h)\right)\right] + O\left(\frac{s}{\log s}\right)\\
    & - C_3 s\left(\left|\theta-\frac{1}{2}\right|^2 \log s +\left|\theta-\frac{1}{2}\right|+ \left|t-\frac{1}{2}\right|^4\right),
\end{align*} 
where $C_3$ and the implicit constant in the term $O\left(\frac{s}{\log s}\right)$ depend only on $\X$ and $C_3=\frac{\log\E\left[e^\X\right]}{64\pi }$ is admissible.
\end{thm}
We now deduce the following corollaries, which will be key ingredients in the proofs of Theorems~\ref{THM Ncal_m(V)} and \ref{thm structure simp} concerning the distribution of the maximum of partial sums.

\begin{cor}\label{THM asymp  uniform (0,1)^2} For all $s\geq \pi $,
\begin{equation}
    \sup_{\theta \in[0,1], t\in [0,1]}\left(
\log \E \left[\exp \Big(s \cdot \sum_{h\in \Z}
     \gamma(h,t,\theta)\X(h)\Big) \right] \right) = 
     \frac{\Mp}{\pi} s \log s+\kappa s+ O\left(\frac{s}{\log s}\right).
\end{equation}

\end{cor}

\begin{proof} The lower bound follows from \eqref{eq asymp 1/2 1/2} and the uniform upper bound by Theorem \ref{THM asymp t theta}.
    
\end{proof}

\begin{cor}\label{Cor asymp T Theta} There exists $C_3>0$ such that uniformly for $\Delta \in [0,1/2]$ and for $s\geq \pi$, we have
\begin{equation}\label{eq unif asymp T Theta}
\sup_{(t,\theta) \in E}
\left(
\log 
\E \left[\exp \Big(s \cdot \sum_{h\in \Z}
     \gamma(h,t,\theta)\X(h)\Big) \right]
\right)
     \leq
     \frac{\Mp}{\pi} s \log s+(\kappa-C_3 \Delta^4) s+ O\left(\frac{s}{\log s }\right),
\end{equation}
where $E = [0,1]^2\setminus[1/2-\Delta,1/2+\Delta]\times[1/2-\Delta^4,1/2+\Delta^4]$.

\end{cor}
\begin{proof} This is an immediate consequence of Theorem \ref{THM asymp t theta}.
\end{proof}

\subsection{Proof of Theorem \ref{THM asymp t theta}}\label{section 2.3 preuve thm 2.2}
To prove Theorem \ref{THM asymp t theta}, we need to establish the following propositions.

\begin{pro}\label{pro conv methode} 
Let $g:\R\to \R_+$ be a positive, convex, even function of class $\Ccal^3 $ such that $g(0)=0.$
Let $(a_n)_{n\in\Z^*}\in [0,1]^{\Z^*}$ be a real sequence, then we have
\[
\sum_{k\in\Z} g\left(\frac{1}{2k+1}\right)-
\sum_{h\in\Z^*} g\left(\frac{a_h}{|h|}\right)
\geq 
\sum_{N = 1}^{\infty} S_{2N} \left[
(2N-1)g\left(\frac{1}{2N-1}\right)
-(2N+1)g\left(\frac{1}{2N+1}\right)\right].
\]
where
\[
S_{2N} = \sum_{n=1}^N 
\frac{2}{2n-1} +\frac{a_{2N}+a_{-2N}}{4N} -
\sum_{n=1}^{2N} 
\frac{a_{n}+a_{-n}}{n}.
\]
Moreover, all the sums converge absolutely.
\end{pro}
We will use Proposition \ref{pro conv methode} with $g: x\mapsto \log \E[\exp(sx\X/\pi)]$ and
\begin{equation}\label{def a_h}
a_{h}(t, \theta) = \pi|h|\gamma(h,t,\theta) = |\cos(\pi \theta+\pi ht)\sin(\pi h t)|\in[0,1],
\end{equation}
for $t, \theta\in[0,1]$ and $h \in \Z^*$.
In fact, we have good properties for $x\mapsto \log \E[\exp(\lambda x\X)]$ for $\lambda>0 $.
\begin{lem}\label{lem log E exp convexe} For all $\lambda>0,$ the function $g (x):= \log \E[\exp(\lambda x\X)]$ is even, increasing on $(\R^+)^*$ and convex with $g(0) = 0.$
\end{lem}
\begin{proof} Since $\X$ is symmetric, $g$ is even and $g(0) =0.$
For all $x>0,$ we have
\[
g'(x) = \frac{\lambda\E[\X\exp(\lambda x\X)]}{\E[\exp(\lambda x\X)]}\geq 0,
\]
and
\[
g''(x) = \lambda^2\frac{\E[\X^2\exp(\lambda x\X)]\E[\exp(\lambda x\X)]-\E[\X\exp(\lambda x\X)]^2}{\E[\exp(\lambda x\X)]^2}\geq 0,\]
since, by the Cauchy--Schwarz inequality,
$
\E[\X\exp(\lambda x\X)]^2
\leq
\E[\X^2\exp(\lambda x\X)]
\E[\exp(\lambda x\X)].$
\end{proof}
Now, we have the following proposition that gives a lower bound for the sum $S_{2N}$ which is sufficient to prove Theorem \ref{THM asymp t theta}. We define
\begin{equation}\label{def S 2N}
    S_{2N}(t,\theta) := \sum_{n=1}^N 
\frac{2}{2n-1}+\frac{a_{2N}(t,\theta)+a_{-2N}(t,\theta)}{4N} -
\sum_{n=1}^{2N} 
\frac{a_n(t,\theta)+a_{-n}(t,\theta)}{n},
\end{equation}
where $a_n(t,\theta)$ is defined in \eqref{def a_h}.

\begin{pro}\label{thm S_2N > marge (t,theta)}
Let \(N\geq1\) be an integer and let \(t,\theta\in[0,1]\). We have
\[
    S_{2N}(t,\theta)
    \geq  
    \frac{\cos(\pi\theta)^2}{128}\log(3\pi(2N-1))
    +\sin^4\left(\frac{\pi}{3}(t-1/2)\right).
    \]
\end{pro}

\begin{proof}[Proof of Theorem \ref{THM asymp t theta}]
Let $s>0$ be a real number, put $g(x) = \log\E\left[\exp\left(s x\X/\pi\right)\right]$. Then by Lemma \ref{lem log E exp convexe}, $g$ is a positive convex even  function which is increasing on $\R^+$ and such that $g(0) = 0$. Moreover, $ a_{h} = a_{h}(t, \theta) = \pi|h|\gamma(h,t,\theta)$, thus we can write
\[\sum_{h\in \Z} \log \E \left[\exp \left(s \cdot  \gamma(h,t,\theta)\X(h)\right) \right] 
= \log \E \left[\exp \left(s \cdot  \gamma(0,t,\theta)\X(0)\right) \right]
+\sum_{h \in\Z^*} g \left(\frac{a_h(t,\theta)}{|h|
}\right).
\] 
Now, since $a_h\in[0,1],$ we can use Proposition \ref{pro conv methode} to obtain
\begin{align*}&\sum_{h \in\Z^*} g \left(\frac{a_h}{|h|}\right)-
\sum_{k\in\Z} g\left(\frac{1}{|2k+1|}\right) \\
&\leq 
-\sum_{N = 1}^{\infty} S_{2N}(t,\theta) \left[
(2N-1)g\left(\frac{1}{2N-1}\right)
-(2N+1)g\left(\frac{1}{2N+1}\right)\right].
\end{align*}
By \eqref{eq gamma(h,1/2,1/2)}, we obtain
\[
\log \E \left[\exp \left(s \cdot \sum_{h\in \Z} \gamma\left(h,\frac{1}{2},\frac{1}{2}\right)\X(h)\right)\right]
= \log \E \left[\exp \left(s \cdot \sum_{k\in \Z} \frac{\X(2k+1)}{|2k+1|\pi}\right)\right] = \sum_{k\in\Z} g\left(\frac{1}{|2k+1|}\right).\]
On the other hand, by Proposition \ref{thm S_2N > marge (t,theta)}, we have the lower bound,
\[S_{2N}(t,\theta)
    \geq  
    \frac{\cos(\pi\theta)^2}{128}\log(3\pi(2N-1))
    +\sin^4\left(\frac{\pi}{3}(t-1/2)\right) 
    \geq\lambda_1+\lambda_2 \log(3\pi(2N-1)),
\]
where for all $t,\theta\in[0,1],$
\[
\lambda_1 :=  \frac{1}{16}\left|t-\frac{1}{2}\right|^4
\leq 
    \sin^4\left(\frac{\pi}{3}(t-1/2)\right), \qquad
\lambda_2 :=   
    \frac{\cos(\pi\theta)^2}{128}
    \geq 0. 
\]
Moreover, since $g(0)=0$ and $g$ is convex, $\forall \mu\in[0,1],\, \forall x\in\R , \, \mu g( x)\geq  g(\mu x), $
thus we have
\[(2N-1)g\left(\frac{1}{2N-1}\right)
-(2N+1)g\left(\frac{1}{2N+1}\right)
\geq 0.
\]
Since $s \geq \pi $, let \(N_0 \in \mathbb{N}^*\) be such that
$2N_0-1 \leq s/\pi < 2N_0+1.$
Summing these terms gives
\begin{align*}
    &\sum_{N = 1}^{\infty} S_{2N}(t,\theta) \left[
(2N-1)g\left(\frac{1}{2N-1}\right)
-(2N+1)g\left(\frac{1}{2N+1}\right)\right]\\
\geq& \sum_{N = 1}^{\infty}  \lambda_1\left[
(2N-1)g\left(\frac{1}{2N-1}\right)
-(2N+1)g\left(\frac{1}{2N+1}\right)\right]\\
&+\sum_{N = N_0}^{\infty}  \lambda_2 \log(3\pi(2N-1))\left[
(2N-1)g\left(\frac{1}{2N-1}\right)
-(2N+1)g\left(\frac{1}{2N+1}\right)\right]\\
\geq & \lambda_1 g(1)+\lambda_2 \log(3\pi(2N_0-1))
(2N_0-1)g\left(\frac{1}{2N_0-1}\right).
\end{align*}
Now, since $\log(3\pi(2N_0-1))\geq \log(\pi(2N_0+1)) \geq \log s$ and since $g$ is convex, we have
$(2N_0-1)g\left(\frac{1}{2N_0-1}\right)\geq \frac{s}{\pi} g\left(\frac{\pi }{s}\right) $ and hence
\begin{align*}
&\sum_{N = 1}^{\infty} S_{2N}(t,\theta) \left[
(2N-1)g\left(\frac{1}{2N-1}\right)
-(2N+1)g\left(\frac{1}{2N+1}\right)\right]\\
&\geq  \lambda_1 g(1)+\lambda_2(\log s) \frac{s}{\pi} g\left(\frac{\pi}{s}\right)\\
&\geq \frac{g(1)}{16}\left|t-\frac{1}{2}\right|^4+
\frac{\cos(\pi\theta)^2}{128 \pi }g\left(\frac{\pi}{s}\right) s\log s 
    .
\end{align*}
We have $g(\pi/s)  = \log\E\left[e^\X\right]$ and by \cite[Lemma 5.3]{ABL21}, we obtain $$g(1)=\log\E\left[\exp\left(s \X/\pi\right)\right] = \Mp s/\pi+O(\log s),$$ and
\[
\log \E \left[\exp \left(s \cdot  \gamma(0,t,\theta)\X(0)\right) \right] =\log\E\left[\exp\left(s t|\cos(\pi\theta)|\X\right)\right] \leq \Mp st|\cos(\pi\theta)|.
\]
Thus, combining these results, we have shown
\begin{align*}
&\log \E \left[\exp \left(s \cdot \sum_{h\in \Z} \gamma(h,t,\theta)\X(h)\right) \right] 
    -  \log \E \left[\exp \left(s \cdot \sum_{h\in \Z} \gamma\left(h,\frac{1}{2},\frac{1}{2}\right)\X(h)\right)\right]\\
    &\leq \Mp st|\cos(\pi\theta)| - \frac{\Mp s}{16 \pi }\left|t-\frac{1}{2}\right|^4 
    - \frac{\cos(\pi\theta)^2}{128 \pi }\log\E\left[e^\X\right] s\log s +O(\log s) \\
    &\leq -\Mp |\cos(\pi\theta)|s
     -\frac{\Mp }{16 \pi}\left|t-\frac{1}{2}\right|^4 s -C_3 \frac{\cos(\pi\theta)^2}{4}s\log s
     + A(t,\theta,s)+O\left(\log s\right).
\end{align*} 
where $C_3 := \frac{\log\E\left[e^\X\right]}{64\pi}$ and
\[
A(t,\theta,s) \leq  2\Mp s|\cos(\pi\theta )|
-C_3\frac{\cos(\pi\theta)^2}{4}s\log s
= 
 2\Mp s|\cos(\pi\theta )|
 \left(1-\frac{C_3}{8\Mp}|\cos(\pi\theta )|\log s\right). 
\]
If $1\leq \frac{C_3}{8\Mp}|\cos(\pi\theta)|\log s $ then $A(t,\theta,s)\leq 0,$
otherwise $|\cos(\pi \theta)|\ll \frac{1}{\log s} $ and hence $A(t,\theta,s)\ll s/\log s$ uniformly for $t$ and $\theta$. Hence, by  $|\cos(\pi\theta)| \geq 2|\theta-1/2|$, we obtain
 \begin{align*}
&\log \E \left[\exp \left(s \cdot \sum_{h\in \Z} \gamma(h,t,\theta)\X(h)\right) \right] 
    -  \log \E \left[\exp \left(s \cdot \sum_{h\in \Z} \gamma\left(h,\frac{1}{2},\frac{1}{2}\right)\X(h)\right)\right]\\
    &\leq -2\Mp \left|\theta-\frac{1}{2}\right| s
     -\frac{\Mp }{16 \pi}\left|t-\frac{1}{2}\right|^4 s -C_3 \left|\theta-\frac{1}{2}\right|^2 s\log s
     +O\left(\frac{s}{\log s}\right).
\end{align*} 
Since $\Mp \geq \log\E\left[e^\X\right]$, then $C_3 = \frac{\log\E\left[e^\X\right]}{64\pi}$ is admissible, which completes the proof.
\end{proof}

\subsection{Proof of Proposition \ref{pro conv methode}}

\begin{proof}[Proof of Proposition \ref{pro conv methode}] Let $(a_n)_{n\in\Z^*}\!\in[0,1]^{\Z^*}\!\!,$ $g$ be a positive, convex, even function such that $g(0) = 0$, then, we have the following convex inequalities: 
\begin{equation}\label{ineq conv 1}
    \forall a\in [0,1], \forall t\in\R,\qquad g(a t)\leq a g(t) = (1-a)g(0)+ag(t).
\end{equation}
\begin{equation}\label{ineq conv 2}
    \forall n\in \N^*,\qquad g\left(\frac{1}{2n} \right) 
\leq 
\frac{1}{2n}\cdot \frac{2n-1}{2}g\left(\frac{1}{2n-1}\right)+
\frac{1}{2n}\cdot \frac{2n+1}{2}g\left(\frac{1}{2n+1}\right).
\end{equation}

The sums in the statement of Proposition \ref{pro conv methode} converge absolutely since for all $N\in \N^*$, $S_{2N}\ll \log 2N$ and 
there exists $\alpha>0$ such that for all  $t\in[-1,1],\; g\left(t\right)= \alpha t^2+ o\left(t^3\right)$, so for all $N\in \N^*$,
\begin{equation}\label{ineq telescopic sum}
(2N-1)g\left(\frac{1}{2N-1}\right)
-(2N+1)g\left(\frac{1}{2N+1}\right)
\ll_g \frac{1}{N^2}.
\end{equation}

By \eqref{ineq conv 1} and the parity of $g$, we observe that,
\begin{align*}
\sum_{h\in\Z^*} g\left(\frac{a_h}{|h|}\right)
\leq &
\sum_{h\in\Z^*} a_h g\left(\frac{1}{|h|}\right)=\sum_{n = 0}^\infty \left(\frac{a_{2n+1}+a_{-(2n+1)}}{2n+1}\right)(2n+1)g\left(\frac{1}{2n+1}\right)+
\sum_{n \in \Z^*} a_{2n} \,g\left(\frac{1}{|2n|}\right).
\end{align*}
Now by \eqref{ineq conv 2}, we obtain,
\[
\sum_{n \in \Z^*} a_{2n}\, g\left(\frac{1}{|2n|}\right)
\leq 
\sum_{n=1}^\infty
\frac{a_{2n}+a_{-2n}}{4n} \left((2n-1)g\left(\frac{1}{2n-1}\right)+ (2n+1)g\left(\frac{1}{2n+1}\right)\right).
\]
And
\begin{align*}
    \sum_{h\in\Z^*} g\left(\frac{a_h}{|h|}\right)
\leq &\left[a_{1}+a_{-1}
+\frac{a_{2}+a_{-2}}{4}\right]g\left(1\right)\\
+&
\sum_{n\geq 1} \left[\frac{a_{2n+1}+a_{-(2n+1)}}{2n+1}
+\frac{a_{2n}+a_{-2n}}{4n} +\frac{a_{2n+2}+a_{-(2n+2)}}{4n+4}
\right](2n+1)g\left(\frac{1}{2n+1}\right).
\end{align*}
Hence, we see that
\begin{align*}
     &
\sum_{k\in\Z} g\left(\frac{1}{|2k+1|}\right)-
\sum_{h\in\Z^*} g\left(\frac{a_h}{|h|}\right)\\
\geq &
\left[2-a_{1}-a_{-1}
-\frac{a_{2}+a_{-2}}{4}\right]g\left(1\right)\\
+&
\sum_{n\geq 1} \left[\frac{2-a_{2n+1}-a_{-(2n+1)}}{2n+1}
-\frac{a_{2n}+a_{-2n}}{4n} -\frac{a_{2n+2}+a_{-(2n+2)}}{4n+4}
\right](2n+1)g\left(\frac{1}{2n+1}\right)\\
=& S_2 \cdot g\left(1\right)+
\sum_{n \geq 1} \left[S_{2n+2}-S_{2n}\right](2n+1)g\left(\frac{1}{2n+1}\right)\\
=&\sum_{N \geq 1} S_{2N} \left[
(2N-1)g\left(\frac{1}{2N-1}\right)
-(2N+1)g\left(\frac{1}{2N+1}\right)\right],
\end{align*}
the last equality following from Abel's summation formula, since the series converge.
\end{proof}

\subsection{Proof of Proposition \ref{thm S_2N > marge (t,theta)}}
\subsubsection{Definition and decomposition of $a_h$}
Recall that $a_n(t,\theta)$ is defined in \eqref{def a_h}. Then:
\begin{align*}
   a_{n}(t, \theta) :&= \pi|n|\gamma(n,t,\theta) = |\cos(\pi \theta+\pi nt)\sin(\pi n t)|\\
    &=\frac{1}{2}\left(\cos^2(\pi \theta+\pi n t)+\sin^2(\pi n t)-\left(|\sin(\pi n t)|-|\cos(\pi \theta+\pi n t)|
    \right)^2\right)\\
    &= \frac{1}{4}\left(2+\cos(2\pi (\theta+ n t))-\cos(2\pi n t)\right)-\frac{\Delta( n t,\theta)}{2},
\end{align*}
where
\begin{equation}\label{def Delta (x,theta)}
\Delta(x,\theta):= (|\sin(\pi x)|-|\cos(\pi (\theta+x))|)^2,
\qquad (x,\theta)\in\R^2.    
\end{equation}
Thus, we decompose $S_{2N}$ as follows.
\begin{lem}\label{lem S_2k decomposition}Let \(N\geq1\) be an integer and let \(t,\theta\in[0,1]\). Using the notation introduced in \eqref{def S 2N} and \eqref{def Delta (x,theta)}, we have
    \begin{align*}
    S_{2N}(t,\theta)&\geq  
    \left[\frac{1-\cos(2\pi \theta)}{2}\right]\left(\sum_{n=1}^{2N} \frac{\cos(2\pi nt)-(-1)^n}{n} - \frac{\cos(4\pi N t)-1}{4N}\right)\\
    &+\left[\frac{1+\cos(2\pi \theta)}{2}\right]\log 2+\sum_{n=1}^{2N-1}
    \frac{\Delta(nt,\theta) +\Delta(-nt,\theta)}{2n}.
\end{align*}
\end{lem}

\begin{proof}[Proof of Lemma \ref{lem S_2k decomposition}]
For all \(t,\theta \in [0,1]\) and all \(n \in \mathbb{N}^*\),
\begin{align*}
    &a_n(t,\theta)+a_{-n}(t,\theta) \\
    &= \frac{1}{2}\left[2-\cos(2\pi nt)+\frac{\cos(2\pi (nt+\theta))+\cos(2\pi (nt-\theta))}{2}
    -\Delta(nt,\theta) -\Delta(-nt,\theta)\right]\\
    &= \frac{1}{2}\left[2-(1-\cos(2\pi\theta))\cos(2\pi nt)
    -\Delta(nt,\theta) -\Delta(-nt,\theta)\right]\\
    &= -\frac{1-\cos(2\pi\theta)}{2}(\cos(2\pi nt)-1)+\frac{1+\cos(2\pi \theta)}{2}
    -\frac{\Delta(nt,\theta) +\Delta(-nt,\theta)}{2}.\\
\end{align*}
Thus, we have the decomposition:
\begin{align}
    S_{2N}(t,\theta)
    &=\sum_{n=1}^{2N} \frac{1-(-1)^n}{n}
    +\frac{a_{2N}(t,\theta)+a_{-2N}(t,\theta)}{4N}-\sum_{n=1}^{2N}
     \frac{a_n(t,\theta)+a_{-n}(t,\theta)}{n} \notag 
    \\
    \label{eq1 S_2N}
    &=\left[\frac{1-\cos(2\pi \theta)}{2}\right]\Bigg(\sum_{n=1}^{2N}\frac{1-(-1)^n}{n}-\frac{\cos(4\pi Nt)-1}{4N}
    +\sum_{n=1}^{2N}\frac{\cos(2\pi nt)-1}{n}\Bigg) 
    \\
    \label{eq2 S_2N}
    &+ \left[\frac{1+\cos(2\pi \theta)}{2}\right]
    \Bigg(\sum_{n=1}^{2N}\frac{1-(-1)^n}{n}+\frac{1}{4N}-
    \sum_{n=1}^{2N}\frac{1}{n}\Bigg) 
    \\
    &+\frac{\Delta(2Nt,\theta) +\Delta(-2Nt,\theta)}{8N}+\sum_{n=1}^{2N-1}
    \frac{\Delta(nt,\theta) +\Delta(-nt,\theta)}{2n}.
    \label{eq3 S_2N}
\end{align}
Now, for \eqref{eq2 S_2N}, since $ \frac{1}{4N}-\frac{1}{4N+4}\geq \frac{1}{2N+1}-\frac{1}{2N+2}$, we have
\[
    \frac{1}{4N}-
    \sum_{n=1}^{2N}\frac{(-1)^n}{n}
\geq
    \frac{1}{4N+4}-
    \sum_{n=1}^{2N+2}\frac{(-1)^n}{n}
\geq  
-\sum_{n=1}^{\infty}\frac{(-1)^n}{n} =
    \log 2 .
\] 
For \eqref{eq3 S_2N}, since $\Delta(2Nt,\theta) +\Delta(-2Nt,\theta)\geq 0$, we conclude the desired bound.
\end{proof}

To prove Proposition \ref{thm S_2N > marge (t,theta)}, it remains to derive lower bounds for \eqref{eq1 S_2N} and \eqref{eq3 S_2N}. For this, we use the following two propositions.
\begin{pro}\label{pro > sin^4}
Let \(N\geq1\) be an integer and let \(t\in[0,1]\). We have
    \[\sum_{n=1}^{2N} \frac{\cos(2\pi nt)-(-1)^n}{n} - \frac{\cos(4\pi N t)-1}{4N}
    \geq \sin^4\left(\frac{\pi}{3}(t-1/2)\right).\]
\end{pro}
\begin{pro}\label{pro Delta(ht)> C}
Let \(N\geq1\) be an integer and let \(t,\theta\in[0,1]\). Using the notation introduced in \eqref{def Delta (x,theta)}, we have
\[
\sum_{n=1}^N\frac{\Delta(nt,\theta)+\Delta(-nt,\theta)}{n}
\geq \frac{1+\cos(2\pi\theta)}{128}(\log N -4).\]
\end{pro}

\begin{proof}[Proof of Proposition \ref{thm S_2N > marge (t,theta)}]
First, combining Lemma \ref{lem S_2k decomposition},
Proposition \ref{pro > sin^4} and
Proposition \ref{pro Delta(ht)> C}, we obtain the following lower bound for all $t, \theta\in [0,1]$,
\begin{align*}
    S_{2N}(t,\theta)
    &\geq  
    \left[\frac{1+\cos(2\pi \theta)}{2}\right]\left(\log 2 +\frac{1}{128}(\log(2N-1)-4)\right)
    +\left[\frac{1-\cos(2\pi \theta)}{2}\right]\sin^4\left(\frac{\pi}{3}(t-1/2)\right).
\end{align*}
Since $\sin^4\left(\frac{\pi}{3}(t-1/2)\right)\leq \sin^4(\pi/6) = 1/16$ and $1/16+(4+\log 3\pi)/128\leq \log 2$, we obtain 
\begin{align*}
    S_{2N}(t,\theta)
    &\geq  
    \frac{\cos(\pi\theta)^2}{128}\log(3\pi(2N-1))
    +\sin^4\left(\frac{\pi}{3}(t-1/2)\right),
\end{align*} as desired.
\end{proof}

\subsubsection{Proof of Proposition \ref{pro > sin^4}}
In this subsection, we show a slight improvement of \cite[Lemma 4.2]{ABL21} that allows us to get a small gain when $t\ne 1/2$. 
To prove Proposition~\ref{pro > sin^4}, it suffices to combine the following lemmas:
\begin{lem} Let \(N\geq1\) be an integer and let \(t\in[0,1]\). We have
    \[
    \sum_{n=1}^{2N} \frac{\cos(2\pi nt)-(-1)^n}{n} - \frac{\cos(4\pi N t)-1}{4N}
    =2\pi \int_{0}^{|t-1/2|} \frac{\sin(2N\pi u)^2\sin(\pi u)}{\cos(\pi u)}du.
    \]
\end{lem}
\begin{proof}
     Let us define for $t \in [0,1]$, 
     \[
     S(t) = \sum_{n=1}^{2N} \frac{\cos(2\pi nt)-(-1)^n}{n} - \frac{\cos(4\pi N t)-1}{4N}.
     \]
     Then, as in the proof of \cite[Lemma 4.2]{ABL21}, we determine the sign of \(S'\) 
\begin{align*}
     S'(t)&= \sum_{n=1}^{2N} -2\pi \sin(2\pi nt)+ \pi\sin(4\pi N t)\\
     &=  -2\pi\frac{\sin(2\pi N t)\sin(\pi(2N+1)t) }{\sin(\pi t)}+ 2\pi\sin(2\pi N t)\cos(2\pi Nt)\\
     &=  -2\pi\frac{\sin(2\pi N t)^2\cos(\pi t)
     }{\sin(\pi t)}.\\
\end{align*}

Now, we can conclude since $S(1/2) = 0$ and the change of variable $u = x-1/2$ gives
\[
S(t) = S(1/2)+\int_{1/2}^{t} -2\pi\frac{\sin(2\pi N x)^2\cos(\pi x)}{\sin(\pi x)} dx = 2\pi\int_{0}^{t-1/2} \frac{\sin(2\pi N u)^2\sin(\pi u)}{\cos(\pi u)} du\geq 0.
\]
The parity concludes the proof.
\end{proof}

\begin{lem}Let \(N\geq1\) be an integer and let \(U\in[0,1/2]\). We have
    \[\int_{0}^U \frac{\sin(2N\pi u)^2\sin(\pi u)}{\cos(\pi u)}du
    \geq
    \int_{0}^{U/3} 
   \frac{\sin(2\pi u)^2\sin(\pi u)}{2\cos(\pi u)}du
    = \frac{1}{2\pi}\sin^4\left(\frac{\pi U}{3}\right).\]
\end{lem}
\begin{proof}
    If $U\leq \frac{3}{4N}$, then the result follows from $\sin(2N\pi u)^2\geq \sin(2\pi u)^2 $ for $u \in 
    [0,\frac{1}{4N} ] $. 
    \[
    \int_{0}^U \frac{\sin(2N\pi u)^2\sin(\pi u)}{\cos(\pi u)}du\geq 
    \int_{0}^{\min\left(U,\frac{1}{4N}\right)} \frac{\sin(2\pi u)^2\sin(\pi u)}{\cos(\pi u)}du
    \geq\int_{0}^{U/3} 
    \frac{\sin(2\pi u)^2\sin(\pi u)}{\cos(\pi u)} du.
    \]

    If $\frac{3}{4N}\leq U$, let us note that since $T(u) = \tan(\pi u)$ is a convex function on $[0,1/2]$, using that \(\sin^2\) is even and periodic, we have 
    \[
    \int_{\frac{k}{2N}-\frac{1}{4N}}^{\frac{k}{2N}+\frac{1}{4N}}  \sin(2N\pi u)^2 T(u)du = 
    \int_{0}^{\frac{1}{4N}}\sin(2N\pi u)^2 \left[
    T\left(\frac{k}{2N} -u\right)+
    T\left(\frac{k}{2N} +u \right)\right] du \geq \frac{1}{4N}T\left(\frac{k}{2N}\right). 
    \] 
    Since $T $ increases on $\left[0,\frac{1}{2}\right]$, we obtain
    \[
    \int_{\frac{1}{4N}}^{\frac{2K+1}{4N}} \sin(2N\pi u)^2 T(u)du
    \geq
    \sum_{k=1}^{K}\frac{1}{4N}T\left(\frac{k}{2N}\right)
    \geq
    \int_{0}^{\frac{K}{2N}}\frac{1}{2}T\left(u\right)du
    \geq \int_{0}^{\frac{K}{2N}}
    \frac{\sin(2\pi u)^2\sin(\pi u)}{2\cos(\pi u)}du.
    \]
    Now, if $\frac{K}{2N}+\frac{1}{4N}\leq U \leq  \frac{K}{2N}+\frac{3}{4N}$ with $K\geq 1$ then $\frac{K}{2N}\geq  \frac{1}{3}\left(\frac{K}{2N}+\frac{3}{4N}\right)\geq U/3$, we get
    \[
    \int_{0}^{U} \sin(2N\pi u)^2 T(u)du
    \geq
    \int_{0}^{\frac{K}{2N}}
    \frac{\sin(2\pi u)^2\sin(\pi u)}{2\cos(\pi u)}du
    \geq
    \int_{0}^{\frac{U}{3}}
    \frac{\sin(2\pi u)^2\sin(\pi u)}{2\cos(\pi u)}du.
    \]
    We conclude by the following identity, for all $U\in[0,1/2]$, 
    \[\int_{0}^{\frac{U}{3}}
    \frac{\sin(2\pi u)^2\sin(\pi u)}{2\cos(\pi u)}du = 
    \int_{0}^{\frac{U}{3}}
    2\sin(\pi u)^3\cos(\pi u)du
     = \frac{1}{2\pi} \sin^4\left(\frac{\pi U}{3}\right).
    \]
\end{proof}

\subsubsection{Study of $\Delta(x,\theta)$} In this section, we prove Proposition \ref{pro Delta(ht)> C} to get a lower bound for \eqref{eq3 S_2N}.
To this end, we need the following lemma.
\begin{lem}\label{D(x)+D(2x)>theta 2}
Let \(x,\theta\in \R \). Using the notation introduced in \eqref{def Delta (x,theta)} there exists $y \in \{-2x,-x,x,2x \}$, such that
\begin{equation*}
\Delta(y,\theta) 
\geq \frac{\cos(\pi\theta)^2}{4}(2-\sqrt{3}).
\end{equation*}
\end{lem}
\begin{proof}[Proof of Lemma \ref{D(x)+D(2x)>theta 2}]
By 1-periodicity of $\Delta(x,\theta)$ as a function of $x$ and $\theta$,
using the symmetry $\Delta(x,\theta)=\Delta(-x,-\theta)$ and the symmetry of the problem $y \in \{-2x,-x,x,2x \}$, we can assume that $x\in[0,1/2]$ and $\theta = \vartheta+\frac{1}{2}\in[1/2,1]$. Thus we have 
\[ 
\Delta(x,\vartheta+1/2) = (|\sin(\pi x)|-|\sin(\pi (x+\vartheta))|)^2.
\]
Depending on the signs of
\(\sin(\pi x)\) and \(\sin(\pi(x+\vartheta))\),
we obtain one of the following two expressions:
\begin{align}\label{eq  Delta 1}
\big(\sin(\pi x)-\sin(\pi (x+\vartheta))\big)^2
    &=
    \big(2\sin(\pi\vartheta/2)^2 \sin(\pi x)-2\sin(\pi \vartheta/2)\cos(\pi\vartheta /2)\cos(\pi x)\big)^2
    \\
    &=4\sin(\pi\vartheta/2)^2
    \cos(\pi x+\pi\vartheta /2)^2 \notag
    \\
    &= (1-\cos(\pi \vartheta))(1
    +\cos(2\pi x+\pi \vartheta)),\notag
\end{align}
and
\begin{align}\label{eq  Delta 2}
\big(\sin(\pi x)+\sin(\pi (x+\vartheta))\big)^2
&= (1+\cos(\pi \vartheta))(1
    -\cos(2\pi x+\pi \vartheta)).
\end{align}

\underline{Case $1$ :  $\vartheta \in\left[0,\frac{1}{6}\right]$.} If $x\in[1/3,1/2]$ then $-2x\in[-1,-2/3]$, so that, by \(1\)-periodicity, we may assume that \(y\in[0,1/3]\). 
Hence, identity \eqref{eq Delta 1} applies, because $\sin(\pi y)\geq0$ and $\sin(\pi (y+\vartheta))\geq0$. Since $$ (1-\cos(\pi \vartheta))=2\sin(\pi\vartheta/2)^2\geq 2\sin(\pi\vartheta/2)^2\cos(\pi\vartheta/2)^2 =\frac{\sin(\pi\vartheta)^2}{2},$$
we obtain the desired bound,
\begin{align*}
\Delta(y,\vartheta+1/2) &= 
(1-\cos(\pi \vartheta))(1
    +\cos(2\pi y+\pi \vartheta)) \geq
    \frac{\sin(\pi\vartheta)^2}{2}
    (1+\cos(2\pi/3+\pi/6)).
\end{align*}

\underline{Case $2$ : $\vartheta \in\left[\frac{1}{6},\frac{1}{2}\right]$.} 
Since $1\pm\cos(\pi \vartheta)\geq \frac{\sin(\pi\vartheta)^2}{2} =\frac{\sin(\pi\vartheta+\pi)^2}{2}  $, both \eqref{eq Delta 1} and \eqref{eq Delta 2} imply
\[
\Delta(x,\vartheta+1/2) \geq 
\frac{\sin(\pi\vartheta)^2}{2}
(1-|\cos(2\pi x+\pi \vartheta)|).
\]
We choose \(y\in\{x,-x,2x,-2x\}\) such that
\(y \bmod \frac12 \in [0,1/6]\), which implies $2 y+ \vartheta \bmod 1 \in[1/6,1/2+1/3]$. Thus, we conclude
\[
\Delta(y,\vartheta+1/2)\geq 
\frac{\sin(\pi\vartheta)^2}{2}
(1-|\cos(\pi/6)|)
= \frac{\sin(\pi\vartheta)^2}{2}\frac{2-\sqrt{3}}{2}.
\]

\end{proof}

\begin{proof}[Proof of Proposition \ref{pro Delta(ht)> C}] Let us put $C := \frac{\cos(\pi\theta)^2}{16}$. By Lemma \ref{D(x)+D(2x)>theta 2}, for all $\theta, x\in\R,$ we have
\begin{equation}
\Delta(x,\theta)+\Delta(-x,\theta)
+\Delta(2x,\theta)+\Delta(-2x,\theta) 
\geq
    \frac{\cos(\pi\theta)^2}{4}(2-\sqrt{3})\geq  C.
\end{equation} 
Hence, for all $N\geq 1$, since the interval \((N/4,N/2]\) contains at least \(\lfloor N/4 \rfloor\) integers, we have
\[C+\sum_{1\leq |n|\leq N}\Delta(nt,\theta)
\geq C+
\sum_{N/4< n\leq  N/2}\Delta(nt,\theta)+\Delta(-nt,\theta)+\Delta(2nt,\theta)+\Delta(-2nt,\theta)
\geq 
\frac{CN}{4}.\] We define
$\delta_1 = C+\Delta(t,\theta)+\Delta(-t,\theta)$ and for $n\geq 2$, $\delta_n = \Delta(nt,\theta)+\Delta(-nt, \theta)$. 
Using summation by parts, since
$\sum_{n=1}^k \delta_n \geq Ck/4,$
we obtain
\begin{align*}
\sum_{n=1}^N \frac{\delta_n}{n} = \frac{1}{N}\sum_{n=1}^N \delta_n
+\sum_{k=1}^{N-1}\frac{1}{k(k+1)}\sum_{n=1}^k \delta_n\geq \frac{C}{4} \left( 1
+\sum_{k=1}^{N-1}\frac{1}{k+1}\right)
\geq C \frac{\log N}{4}.
\end{align*}
Therefore,
\[
\sum_{n=1}^N\frac{\Delta(nt,\theta)+\Delta(-nt,\theta)}{n}
\geq C
\left(\frac{\log N}{4}-1\right)
=
\frac{1+\cos(2\pi\theta)}{128}(\log N-4),\]
for all $N\geq 1$, as desired.
\end{proof}

\section{Distribution of the uniform maximum of the partial sums }

\subsection{Preliminary lemmas from \cite{ABL21}} 
To prove Theorem~\ref{THM Ncal_m(V)} and Theorem~\ref{thm structure simp}, we follow the approach of Autissier, Bonolis, and Lamzouri in \cite{ABL21}. We will make use of several of their lemmas; these depend only on Assumptions~2 and~3 of \cite{ABL21}, both of which are implied by Assumption~(\hyperlink{Ass Law}{Law}).

\begin{lem}\cite[Lemma 6.1]{ABL21}\label{lemme 6.3}
Let $m$ be a large enough integer and $1\leq y<z \leq m/2$ be real numbers. Let $\{\varphi_a\}_{ a \in\Omega_m} $ be a family of $m-$periodic complex-valued functions satisfying Assumption (\hyperlink{Ass Law}{Law}). Let $\{c(h)\}_{ h\in\Z^*} $ be a sequence of complex numbers such that $|c(h)| \leq c_0/|h|$
for $|h| \geq 1$, where $c_0$ is a positive constant.
Then, for all positive integers $k \leq \log m/(5 \log \log m),$ we have:
\[
\frac{1}{|\Omega_m|}
\sum_{a\in \Omega_m}
\left|
\sum_{y \leq|h|<z}
c(h)\widehat{\varphi_a}(h)
\right|^{2k }
\ll
\left(
\frac{16(c_0 \Mp)^2 k}{y}
\right)^k
+\frac{(4C_1 c_0\log m)^{2k}}{m^{1/2}},
\]
where the constant implied by the $\ll$ notation is absolute.
\end{lem}

The following lemma is a generalization of \cite[Proposition 2.2]{Lam20}. In \cite{ABL21}, it is stated for \(y\geq1\) (see the proof of \cite[Lemma 6.1]{ABL21}). Under Assumption~(\hyperlink{Ass Law}{Law}), this restriction can be removed, allowing us to consider \(y \geq 0\), and we obtain the following form.
\begin{lem}\label{Pro 2.2} Let $m$ be a large enough integer. Let $\{\varphi_a\}_{ a \in\Omega_m} $ be a family of $m-$periodic complex-valued functions satisfying Assumption (\hyperlink{Ass Law}{Law}).
Let $\{c(h)\}_{h\in\Z}$ be a sequence of complex numbers.
Let $0 \leq y < z \leq m/2$ be real numbers and let $k, l$ be non-negative integers such that $1\leq k+l \leq \log m/\log\log m$. 
Then, we have
\[\frac{1}{|\Omega_m|}
\sum_{a\in \Omega_m}
\left(
\sum_{y \leq|h|<z}
c(h)\widehat{\varphi_a}(h)
\right)^k 
\left(
\sum_{y \leq|h|<z}
\overline{c(h)}\widehat{\varphi_a}(h)\right)^l
\]\[
= \E
\left[
\left(
\sum_{y \leq|h|<z}
c(h)\X(h)
\right)^k
\left(
\sum_{y \leq|h|<z}
\overline{c(h)}\X(h)
\right)^l
\right]
+ O
\left(
m^{-1/2} \left(C_1 \sum_{y \leq |h|<z} |c(h)|
\right)^{k+l} \right),\]
where $(\X(h))_{h\in\Z}$ denotes independent and identically distributed random variables with the same law as $\X$ in Assumption (\hyperlink{Ass Law}{Law}) and the constant implied by the $O$ notation is absolute.
\end{lem}

\begin{lem}\cite[Lemma 3.1]{Lam20}, \cite[Lemma 5.1]{ABL21} \label{Lemme 3.1}
    Let $\{c(h)\}_{h\in\Z^*}$ be a sequence of complex numbers such that $|c(h)| \leq c_0/|h|$
for $|h| \geq 1$, where $c_0$ is a positive constant. Let 
$(\X(h))_{h\in\Z^*}$ denote independent and identically distributed random variables with the same law as $\X$ in Assumption (\hyperlink{Ass Law}{Law}). 
Let $1 \leq y < z$ be real numbers. Then, for all
positive integers $k$, we have
\[\E \left[ \left|
\sum_{
y \leq|h|<z}
c(h)\X(h)
\right|^k \right]
\leq   \left(\frac{8(c_0 \Mp )^2 k}{y}\right)^{k/2} .\]
Moreover, if $k > y$ then
\[\E
\left[
\left|
\sum_{
y \leq |h|<z}
c(h)\X(h)
\right|^k\right]
\leq ( 10 \Mp c_0 \log k)^k. \]
\end{lem}

As for Lemma~\ref{Pro 2.2}, we shall need the second assertion of Lemma~\ref{Lemme 3.1} for $y=0$, that is, for sums which include the index $h=0$. This makes sense under Assumption~(\hyperlink{Ass Law}{Law}), where the sequence $(\X(h))_{h\in\Z}$ is indexed by all of $\Z$, and the coefficient $c(0)$ is simply absorbed into the trivial bound for the low range, so that the constant is unchanged. We only state the case $c_0=1/2$, which is the one we shall use.

\begin{lem}\label{lem:moments h=0}
Let $\{c(h)\}_{h\in\Z}$ be a sequence of complex numbers such that
$|c(0)|\leq 2$ and $|c(h)|\leq \frac{1}{2|h|}$ for $|h|\geq 1$. Let
$(\X(h))_{h\in\Z}$ denote independent and identically distributed random
variables with the same law as $\X$ in Assumption (\hyperlink{Ass Law}{Law}).
Then, for every real number $z>1$ and every integer $k\geq 4$, we have
\[
  \E\left[\left|\sum_{|h|<z}c(h)\X(h)\right|^{k}\right]
  \leq \left(5\Mp\log k\right)^{k}.
\]
\end{lem}

\begin{proof} By Minkowski's inequality,
\[
  \E\left[\left|\sum_{|h|<z}c(h)\X(h)\right|^{k}\right]^{1/k}
  \leq \E\left[\left|\sum_{|h|<\min(k,z)}c(h)\X(h)\right|^{k}\right]^{1/k}
  + \E\left[\left|\sum_{k\leq|h|<z}c(h)\X(h)\right|^{k}\right]^{1/k},
\]
where the last term is empty if $z\leq k$. Since $|\X(h)|\leq \Mp$ almost surely, the
first term is bounded trivially by
\[
  \Mp \sum_{|h|<k}|c(h)|
  \leq \Mp\left(2+\sum_{1\leq h<k}\frac1h\right)\leq \Mp\left(3+\log k\right).
\]
The second term is bounded by the first assertion of Lemma~\ref{Lemme 3.1}, applied with
$c_0=1/2$ and $y=k\geq 1$, which gives $\big(8(c_0\Mp)^2k/k\big)^{1/2}=\sqrt{2}\,\Mp$.
Altogether,
\[
  \E\left[\left|\sum_{|h|<z}c(h)\X(h)\right|^{k}\right]
  \leq \left(\Mp\left(\log k+3+\sqrt2\right)\right)^k\leq \left(5\Mp\log k \right)^k,
\]
since $4\log k\geq 4\log 4>3+\sqrt2$ for $k\geq 4$.
\end{proof}

\subsection{Discretization and exceptional elements $a\in \Omega_m$}

In this section, we assume Assumptions~(\hyperlink{Ass Law}{Law}) and (\hyperlink{Ass Tightness}{Tightness}) to prove Theorem~\ref{THM Ncal_m(V)} and Theorem~\ref{thm structure simp}. To begin, by the discrete Plancherel formula, the partial sums $S_m(a,t)$ defined in \eqref{def S_m(a,t)} admit the representation \eqref{S_m(a,t) fourier transform}. For $a\in\Omega_m$ and $t\in[0,1]$, this gives
\begin{equation}\label{eq term pair h=m/2}
\B_m(a,t)= \sum_{h\in \mathbb{Z}/m\mathbb{Z}}
\gamma_m(h,t)\widehat{\varphi_a}(h)
= \sum_{|h|< m/2} 
\gamma_m(h,t)\widehat{\varphi_a}(h) + O\!\left( \frac{1}{m} \right),
\end{equation}
where the error term comes from the contribution of \(h=m/2\). More generally, for \(t,\theta \in [0,1]\), 
\begin{equation}\label{eq term pair h=m/2 avec theta}
\Re\!\left(e^{\pi i\theta}\B_m(a,t)\right)
= \sum_{h\in \mathbb{Z}/m\mathbb{Z}}
\gamma_m(h,t,\theta)\widehat{\varphi_a}(h)
= \sum_{|h|< m/2} 
\gamma_m(h,t,\theta)\widehat{\varphi_a}(h) + O\!\left( \frac{1}{m} \right),
\end{equation}
where \(\gamma_m(h,t,\theta):= \Re\!\left(e^{\pi i\theta}\gamma_m(h,t)\right)\).
Moreover, for all \(1\leq |h|< m/2\) and \(t\in[0,1]\), we will use the following bound
\begin{equation}\label{Bound gamma p}
|\gamma_m(h,t)| \leq \frac{1}{m|\sin(\pi h/m)|}\leq \frac{1}{2|h|}.
\end{equation}

\subsubsection{Definitions}
Let \(0<\tau<1\), and define the grid
\[
\Scal_\tau := \left\{ n\tau \; ; \; 1\leq n \leq \frac{1}{\tau} \right\}.
\]
Let $0<\delta< 1/4 $ and $\alpha \geq 1 $ denote the constants in Assumption (\hyperlink{Ass Tightness}{Tightness}).
Set $\beta := \frac{5}{2\delta}-1$. We define the following sets:
\begin{align*}
E^0_m 
&:=\left\{a\in\Omega_m \;:\;
\max_{|t-t'|\leq m^{-1/2+\delta}}\left|
\frac{1}{\sqrt{m}}\sum_{mt' < n\leq mt} \varphi_a(n)
\right| >
m^{-\frac{\delta}{2\alpha}}
\right\},\\
E^1_m &:= \left\{ a \in  \Omega_m: \max _{t\in \Scal_{m^{-1/2+\delta}}} \left|\sum_{(\log m)^{\beta} \leq |h|< m/2} \gamma_m(h,t)\widehat{\varphi_a}(h)\right|> \frac{1}{\log\log m}\right\},\\
E^2_m &:= \left\{ a \in  \Omega_m: \max _{t\in \Scal_{(\log m)^{-\beta-1}}} \left|\sum_{(\log m)^2 \leq |h|<(\log m)^{\beta}} \gamma_m(h,t)\widehat{\varphi_a}(h)\right|> \frac{1}{\log\log m}\right\}.
\end{align*}
Let $H_0 \leq (\log m)^2$ be a large positive integer to choose later in the proof of Proposition \ref{Pro final canonic unif T,Theta}. Let $J:= \left\lfloor \log\left(\frac{(\log m)^2}{H_0}\right) /\log 2\right\rfloor$. We define for all $j \in [\![0,J]\!]  $, $H_j := 2^j H_0 \leq (\log m)^2$ and 
\[E_m(H_j) := 
\left\{ a \in  \Omega_m: \max _{t\in \Scal_{H_j^{-4/3}}} \left|\sum_{H_j\leq |h|< H_{j+1}} \gamma_m(h,t)\widehat{\varphi_a}(h)\right|> \frac{1}{(\log H_j)^2}\right\},
\]
where $H_{J+1}:= (\log m)^2 \leq 2 H_J$
and for all $j \in [\![0,J]\!]$,
\begin{equation}\label{eq def Ecal_m(H_j)}
\Ecal_m(H_j) := E^0_m\cup E^1_m\cup E^2_m\cup
\bigcup_{j \leq  i\leq J}
E_m(H_i).
\end{equation}

Moreover, we should note that all implicit constants can depend on the parameters $\Mp, \alpha, C_1,\delta$,
$c_0 = \frac{1}{2}$ and $\eta= \frac{1}{2}$ where \(\eta\) denotes the exponent of \(m\) appearing in the error term in Assumption (\hyperlink{Ass Law}{Law}) and in Assumption $3$ in \cite{ABL21} (see Proposition \ref{Pro Assumptions implication}). Furthermore, the implicit constants will not depend on $t$ or $\theta,$ so all estimates will be uniform for $t\in[0,1]$ and $\theta\in[0,1]$.

\subsubsection{Size of the sets $E^0_m, E^1_m, E^2_m$ and $E_m(H_j)$, for $0\leq j\leq J$}
The following lemma shows that these sets are not too large.
\begin{lem}\label{lem card set E(H)} Let $m$ be a large enough integer. By assuming Assumptions (\hyperlink{Ass Law}{Law}) and (\hyperlink{Ass Tightness}{Tightness}), we have 
\begin{equation}\label{eq bound E^1,2,3}
\frac{|E^0_m \cup E^1_m\cup E^2_m|}{|\Omega_m|}\ll m^{-\delta/2}.
\end{equation}

\begin{equation}\label{eq bound E(H j)}
\frac{|E_m(H_j)|}{|\Omega_m|}\ll 
 \exp\left(-\frac{H_j^{1/2}}{8}\right), \quad\text{for  } j \in [\![ 0,J ]\!]  .
\end{equation}
\begin{equation}\label{eq bound Ecal(H_0)}
\frac{|\Ecal_m(H_j)|}{|\Omega_m|}\ll 
\exp\left(-\frac{\delta H_j^{1/2}}{2}\right), \quad\text{for  } j \in [\![ 0,J ]\!]  .
\end{equation}
\end{lem}

\begin{proof} 

By Assumption (\hyperlink{Ass Tightness}{Tightness}), there exists $\alpha\geq 1$ such that:
\[
\frac{1}{|\Omega_m|}\sum_{a\in\Omega_m}
\max_{|I|\leq m^{1/2+\delta}}\left|
\frac{1}{\sqrt{m}}\sum_{n\in I} \varphi_a(n)
\right|^{\alpha } \ll
m^{-\delta}.
\]
Therefore
\begin{equation}\label{Fm/Omega <<}
\frac{|E^0_m|}{|\Omega_m|}
\leq  \frac{1}{|\Omega_m|} \sum_{a\in\Omega_m}\left( m^{\frac{\delta}{2\alpha}}
\max_{|t-t'|\leq m^{-1/2+\delta}}\left|
\frac{1}{\sqrt{m}}\sum_{mt' < n\leq mt} \varphi_a(n)
\right| \right)^\alpha
\ll m^{-\delta/2}.
\end{equation}

Now, we will only use Assumption (\hyperlink{Ass Law}{Law}) to bound the size of the remaining sets. Let us take $ r = \left\lfloor\frac{\delta \log m}{5\log\log m}\right\rfloor.$
\begin{align*}
| E^1_m| &= 
\left|\left\{ a \in  \Omega_m: \max _{t\in \Scal_{m^{-1/2+\delta}}} \left|\sum_{(\log m)^{\beta} \leq |h|< m/2} \gamma_m(h,t)\widehat{\varphi_a}(h)\right|> \frac{1}{\log\log m}\right\}\right|
\\
&\leq (\log\log m)^{2r}\sum_{a\in \Omega_m}
\max _{t\in \Scal_{m^{-1/2+\delta}}}\left|
\sum_{(\log m)^{\beta}\leq |h|< m/2}
\gamma_m(h,t)\widehat{\varphi_a}(h)
\right|^{2r} \\
&\ll_\varepsilon m^\varepsilon \sum_{t\in \Scal_{m^{-1/2+\delta}}}\sum_{a\in \Omega_m}
\left|
\sum_
{(\log m)^{\beta}\leq |h|< m/2}
\gamma_m(h,t)\widehat{\varphi_a}(h)
\right|^{2r}.
\end{align*}
By \eqref{Bound gamma p}, we use Lemma \ref{lemme 6.3} with $c_0=1/2$.
There exist two constants $A$ and $B$ such that for all $t\in [0,1]$
\begin{align*}
\frac{1}{|\Omega_m|}
\sum_{a\in \Omega_m}
\left|
\sum_
{(\log m)^{\beta}\leq |h|< m/2}
\gamma_m(h,t)\widehat{\varphi_a}(h)
\right|^{2r}
&\ll
\left(
\frac{A r}{(\log m)^{\beta}}
\right)^r
+\frac{(B\log m)^{2r}}{m^{1/2}}\\
&\ll_\varepsilon m^\varepsilon m^{-\frac{1}{2}+2\delta/5}.
\end{align*}
Since $|\Scal_{m^{-1/2+\delta}}|\leq m^{1/2-\delta} $, we obtain:
\[
\frac{|E^1_m|}{|\Omega_m|}
\ll_\varepsilon 
m^{2\varepsilon}
m^{1/2 -\delta}
m^{-\frac{1}{2}+2\delta/5}
\ll m^{-\delta/2}.
\]
We now bound $E_m^2$, choosing $k = \left\lfloor\frac{\log m }{8\log\log m} \right \rfloor$ for $E^2_m$, we have similarly
\begin{align*}
\frac{| E^2_m|}{|\Omega_m|}
&\leq \frac{(\log\log m)^{2k}}{|\Omega_m|}
\sum_{a\in \Omega_m}
\max _{t\in \Scal_{(\log m)^{-\beta-1}}}\left|
\sum_{(\log m)^{2}\leq |h|< (\log m)^{\beta}}
\gamma_m(h,t)\widehat{\varphi_a}(h)
\right|^{2k} \\
&\ll_\varepsilon m^\varepsilon \sum_{t\in \Scal_{(\log m)^{-\beta-1}}} \frac{1}{|\Omega_m|}\sum_{a\in \Omega_m}
\left|
\sum_
{(\log m)^{2}\leq |h|< (\log m)^{\beta}}
\gamma_m(h,t)\widehat{\varphi_a}(h)
\right|^{2k}
\end{align*}
By Lemma \ref{lemme 6.3} and since \( |\Scal_{(\log m)^{-\beta-1}}| \leq (\log m)^{\beta+1} \) and \(\delta < 1/4\), this yields the desired bound.
\begin{align*}
\frac{| E^2_m|}{|\Omega_m|}
&\ll_\varepsilon m^\varepsilon |\Scal_{(\log m)^{-\beta-1}}|
\left[ 
\left(
\frac{A k}{(\log m)^{2}}
\right)^k
+\frac{(B\log m)^{2k}}{m^{1/2}}\right] \\
&
\ll_\varepsilon m^{2\varepsilon} (\log m)^{\beta+1}
\left[m^{-1/8}+m^{-\frac{1}{2}+2/8}
\right] \\
&\ll m^{-\delta/2}.
\end{align*}
Combining the above results, we obtain \eqref{eq bound E^1,2,3}.

Now, we repeat the same argument to bound $E_m(H_j)$ for $ j \in [\![ 0,J ]\!]  $ with $H_j\leq (\log m)^2.$ Let 
$ l_j = \left\lfloor\frac{H_j^{1/2}}{3\log H_j}
\right\rfloor\leq \frac{\log m}{6\log\log m}.$ By Lemma \ref{lemme 6.3}, we have
\begin{align*}
\frac{| E_m(H_j)|}{|\Omega_m|} &= 
\frac{1}{|\Omega_m|}\left| \left\{
a \in  \Omega_m: 
\max _{t\in \Scal_{H_j^{-4/3}}} 
\left|\sum_{H_j \leq |h|< H_{j+1}} \gamma_m(h,t)\widehat{\varphi_a}(h)\right|
> \frac{1}{(\log H_j)^2}
\right\}\right|
\\
&\leq 
(\log H_j)^{4l_j}
\sum_{t\in \Scal_{H_j^{-4/3}}}
\frac{1}{|\Omega_m|}\sum_{a\in \Omega_m}
\left|
\sum_{H_j \leq |h|< H_{j+1}}
\gamma_m(h,t)\widehat{\varphi_a}(h)
\right|^{2l_j} \\
&\leq  
H_j^{4/3} (\log H_j)^{4l_j}\left[
\left(
\frac{A l_j}{H_j}
\right)^{l_j}
+\frac{(B\log m)^{2l_j}}{m^{1/2}}\right]\\
&\leq H_j^{4/3}
\left(
\frac{A(\log H_j)^3}{4H_j^{1/2}}
\right)^{l_j}
+(\log m)^{8/3}\frac{(2B\log(2\log m)\log m)^{2l_j}}{m^{1/2}}.\\
\end{align*}
Since $ H_j^{1/2}\leq \log m $, we obtain
\[
\frac{| E_m(H_j)|}{|\Omega_m|}
\ll  \exp\left(-\frac{H_j^{1/2}}{8}\right)+
m^{-1/8}\ll 
\exp\left(-\frac{H_j^{1/2}}{8}\right).
\]
This proves \eqref{eq bound E(H j)}. Combining with \eqref{eq bound E^1,2,3}, we obtain \eqref{eq bound Ecal(H_0)}.
\[
\frac{|\Ecal_m(H_j)|}{|\Omega_m|} 
\ll m^{-\delta /2}+
\sum_{i=j}^{J}\exp\left(-\frac{(2^{i-j} H_j)^{1/2}}{8}\right)\]\[
\ll
m^{-\delta /2}
+\exp\left(-\frac{H_j^{1/2}}{8}\right)
\ll
\exp\left(-\frac{\delta H_j^{1/2}}{2}\right).
\]

\end{proof}

\begin{pro}\label{Pro a not in Ep}
Let $H_0$ be a large enough integer such that $ H_0 \leq (\log m)^2, $ and let $a\not\in \Ecal_m(H_j)$ for $0\leq  j \leq J$. Then
\begin{equation} \label{1*}
    \sup_{t\in [0,1]} \left|\B_m(a,t)\right|\leq (\Mp+1)\log H_j .
\end{equation}
Moreover, if $a\not\in \Ecal_m(H_0)$ then we have
\begin{equation} \label{2*}
\underset{|t-t'|\leq H_0^{-4/3}}{\sup_{t,t'\in [0,1]}} \left|\B_m(a,t)-\B_m(a,t')\right|
\ll \frac{1}{\log H_0}.
\end{equation}
\begin{equation} \label{3*}
\underset{|\theta-\theta'|\leq H_0^{-2/3}}{\sup_{t,\theta,\theta'\in [0,1]}} \left|\Re\left(e^{\pi i\theta}\B_m(a,t)\right)-\Re\left(e^{\pi i\theta'}\B_m(a,t)\right)\right|
\ll \frac{1}{\log H_0}.
\end{equation}
\end{pro}

To prove Proposition~\ref{Pro a not in Ep}, we combine the following lemma with a simple discretization argument. For any function \(f:[0,1]\to\mathbb{C}\) and any \(0<\tau<1\), we have
\begin{equation}\label{propri S_rho}
\sup_{t\in[0,1]} |f(t)|
\leq 
\max_{t\in \Scal_{\tau}} |f(t)|
+ \sup_{\substack{t,t'\in [0,1]\\ |t-t'|\leq \tau}} |f(t)-f(t')|.
\end{equation}

\begin{lem}\label{lem gam lip and sum H<h<2H} Let $m$ be an integer. For all $|h|<m/2,$ $\gamma_m(h,\cdot)$ is almost $\pi/2$-Lipschitz, namely, for all $ 0\leq  t' \leq t \leq 1$ we have
\begin{equation}\label{eq almost lipchitz 1}
|\gamma_m(h,t)-\gamma_m(h,t')|
\leq \frac{\pi}{2}\left(|t-t'| +\frac{1}{m} \right).
\end{equation}

\end{lem}

\begin{proof} The case \(h=0\) is trivial, so we may assume that $0<|h|<m/2$ and  $t'\leq t$. Since
$\lfloor mt\rfloor-\lfloor mt'\rfloor\leq
mt-mt'+1.$ Hence
\begin{align*}
|\gamma_m(h,t)-\gamma_m(h,t')|
&=
\frac{1}{m|e_m(h)-1|}\left|
e_m(h\lfloor mt\rfloor)-e_m(h\lfloor mt'\rfloor)
\right|\\
&\leq 
\frac{1}{4|h|}\cdot \frac{2\pi
\left| h\right| }{m}\left(\lfloor mt\rfloor-\lfloor mt'\rfloor\right)
\leq 
\frac{\pi }{2m}\left(mt-mt'+1\right).
\end{align*}
\end{proof}

\begin{proof}[Proof of Proposition \ref{Pro a not in Ep}]
To begin, we recall that:
\begin{equation} \label{(1)}
\sup_{t\in [0,1]} \left|\B_m(a,t)\right|
\leq 
\max_{t\in \Scal_{m^{-1/2+\delta}}} \left|\B_m(a,t)
\right|
+\underset{|t-t'|\leq m^{-1/2+\delta}}{\sup_{t,t'\in [0,1]}} \left|\B_m(a,t)-\B_m(a,t')
\right|.
\end{equation} 
Let \(0 \leq j \leq J\). Assume that \(a \notin \Ecal_m(H_j)\). Then \(a \notin E_m^0\), so we already have
\begin{equation} \label{(2)}
\max_{|t-t'|\leq m^{-1/2+\delta}}\left|
\B_m(a,t)-\B_m(a,t')
\right| \leq
m^{-\frac{\delta}{2\alpha}}.
\end{equation}
Now, by \eqref{eq term pair h=m/2}, $\B_m(a,t) = \sum_{0 \leq|h|< m/2}\gamma_m(h,t)\widehat{\varphi_a}(h) + O(1/m),$ we have
\begin{align}
\sup_{t\in \Scal_{m^{-1/2+\delta}}} \left|\B_m(a,t)\right|
& \leq \label{eq 3.1}
\max_{t\in \Scal_{m^{-1/2+\delta}}} \left|
\sum_{(\log m)^\beta\leq|h|< m/2}\gamma_m(h,t)\widehat{\varphi_a}(h)
\right| + O\left(\frac{1}{m} \right)\\
&\label{eq 3.2}
+\max_{t\in \Scal_{m^{-1/2+\delta}}} \left|
\sum_{(\log m)^2\leq|h|< (\log m)^\beta}\gamma_m(h,t)\widehat{\varphi_a}(h)
\right| \\
\label{eq 3.3}
&+{\sum_{i=j}^{J}}
\max_{t\in \Scal_{m^{-1/2+\delta}}} \left|
\sum_{H_i \leq|h|< H_{i+1}}\gamma_m(h,t)\widehat{\varphi_a}(h)
\right|\\
\label{eq 3.4}
&+\max_{t\in \Scal_{m^{-1/2+\delta}}} \left|
\sum_{|h|<H_j}\gamma_m(h,t)\widehat{\varphi_a}(h)
\right|.
\end{align}
Since \(a \notin E_m^1\), we obtain a bound for the terms on the right-hand side of \eqref{eq 3.1}.
Since \(a \notin E_m^2\), we obtain a bound for \eqref{eq 3.2} by using \eqref{propri S_rho},
\begin{align*}
&\max_{t\in \Scal_{m^{-1/2+\delta}}} \left|
\sum_{(\log m)^2\leq|h|< (\log m)^\beta}\gamma_m(h,t)\widehat{\varphi_a}(h)
\right|\\
&\leq
\max_{
\left|t-t' \right| 
\leq (\log m)^{-\beta-1}} \left|
\sum_{(\log m)^2\leq|h|< (\log m)^\beta}(\gamma_m(h,t)-\gamma_m(h,t'))\widehat{\varphi_a}(h)
\right|+ \frac{1}{\log \log m}\\
&\leq 2(\log m)^\beta \frac{\pi}{2}\left(\frac{1}{(\log m)^{\beta +1} }+ \frac{1}{m} \right)\Mp + \frac{1}{\log \log m}
.
\end{align*}
By \eqref{propri S_rho} and since $a\not\in E_m(H_i)$ for $i \in [\![ j, J]\!] 
$, we obtain the bound of \eqref{eq 3.3}, 
\begin{align*}&\max_{t\in \Scal_{m^{-1/2+\delta}}} \left|
\sum_{H_i \leq|h|< H_{i+1}}\gamma_m(h,t)\widehat{\varphi_a}(h)
\right|\\
&\leq 
\max _{t\in \Scal_{H_i^{-4/3}}} \left|\sum_{H_i \leq |h|< H_{i+1}} \gamma_m(h,t)\widehat{\varphi_a}(h)\right| +
\max _{|t-t'|\leq H_i^{-4/3}} \left|\sum_{H_i \leq |h|< H_{i+1}} (\gamma_m(h,t)-\gamma_m(h,t'))\widehat{\varphi_a}(h)\right|\\
&\leq 
\frac{1}{(\log H_i)^2} +
2H_i \frac{\pi }{2}\left(\frac{1}{H_i^{4/3}}+\frac{1}{m}\right)
\Mp
\ll  \frac{1}{(\log H_i)^2}.
\end{align*}
For \eqref{eq 3.4}, we use the trivial bound with \eqref{Bound gamma p}.
Hence, by combining with \eqref{(1)} and \eqref{(2)},  there exist $c_1>0$ and $c_2>0$ such that
\begin{align*}
\sup_{t\in [0,1]} \left|\B_m(a,t)\right|
&\leq \frac{c_1}{\log\log m} +{\sum_{i=j}^{J}}
\frac{c_2}{(\log H_i)^2}+ \Mp +
\sum_{1\leq |h|<H_j}\frac{\Mp}{2|h|}
.
\end{align*}
Thus, for $H_0$ large enough, we obtain \eqref{1*}.

We follow the same approach to prove \eqref{2*}. By \eqref{(1)} and \eqref{(2)} we have
\begin{align*}
\underset{|t-t'|\leq H_0^{-4/3}}{\sup_{t,t'\in [0,1]}} \left|\B_m(a,t)-\B_m(a,t')\right|
& \leq
2\max_{t\in \Scal_{m^{-1/2+\delta}}} \left|
\sum_{H_0\leq|h|< m/2}\gamma_m(h,t)\widehat{\varphi_a}(h)
\right| + O\left( m^{\frac{-\delta}{2\alpha}}\right)\\
&+\underset{|t-t'|\leq H_0^{-4/3}}{\max_{t,t'\in \Scal_{m^{-1/2+\delta}}}} \left|
\sum_{|h|<H_0}(\gamma_m(h,t)-\gamma_m(h,t'))\widehat{\varphi_a}(h)
\right|.
\end{align*}
Combining \eqref{eq 3.1}, \eqref{eq 3.2}, and \eqref{eq 3.3} with their bounds, we obtain
$$
2\max_{t\in \Scal_{m^{-1/2+\delta}}} \left|
\sum_{H_0\leq|h|< m/2}\gamma_m(h,t)\widehat{\varphi_a}(h)
\right|
\ll \frac{1}{\log\log m} +{\sum_{i=0}^{J}}
\frac{1}{(\log H_i)^2}\ll  \frac{1}{\log H_0}.
$$
Moreover, the last term is controlled using \eqref{eq almost lipchitz 1}. Therefore
\[
\underset{|t-t'|\leq H_0^{-4/3}}{\max_{t,t'\in \Scal_{m^{-1/2+\delta}}}} \left|
\sum_{|h|<H_0}(\gamma_m(h,t)-\gamma_m(h,t'))\widehat{\varphi_a}(h)
\right|
\leq (2H_0 -1)\frac{\pi}{2}\left(H_0^{-4/3}+\frac{1}{m}\right) \Mp
\ll
\frac{1}{H_0^{1/3}} .
\]
Thus, since $ H_0 \leq (\log m)^2$, we conclude that \eqref{2*} holds:
\begin{align*}
\underset{|t-t'|\leq H_0^{-4/3}}{\sup_{t,t'\in [0,1]}} \left|\B_m(a,t)-\B_m(a,t')\right|
\ll \frac{1}{\log H_0}.
\end{align*}

Finally, for $|\theta-\theta'|\leq H_0^{-2/3}$, by \eqref{1*}, we have
\begin{align*}
\left|\Re\left(e^{\pi i\theta}\B_m(a,t)\right)-\Re\left(e^{\pi i\theta'}\B_m(a,t)\right)\right|
&=
\left|\Re\left((e^{\pi i\theta}-e^{\pi i\theta'})\B_m(a,t)\right)\right|
\leq
\pi |\theta-\theta'|\cdot 
\left|\B_m(a,t)\right|\\
&\leq
\pi \frac{(\Mp +1)\log H_0 }{H_0^{2/3}}
\ll \frac{1}{\log H_0}.
\end{align*}
This completes the proof.
\end{proof}

\subsection{Estimating the Laplace transform of the partial sums}

We now adapt Lamzouri's  \cite[Proposition 7.2]{Lam20}, to use in our setting; namely, we have

\begin{pro}\label{pro 7.2 adapted H_0}
Let $m$ be large and $H_0$ be a large enough integer such that $H_0\leq (\log m)^2$. Assume Assumptions (\hyperlink{Ass Law}{Law}) and (\hyperlink{Ass Tightness}{Tightness}). Then there exists a set $\Ecal_m(H_0) \subset \Omega_m$ with cardinality
$$\frac{|\Ecal_m(H_0)|}{|\Omega_m|} \ll \exp\left(-\frac{\delta H_0^{1/2}}{2}\right),$$ 
where $\delta $ is the constant in Assumption (\hyperlink{Ass Tightness}{Tightness}), such that for all complex numbers $s$ with $|s| \leq \frac{\delta H_0^{1/2}}{50(\Mp+1)(\log H_0)^2}$, uniformly for $t,\theta\in [0,1]$, we have
\begin{align*}
    \frac{1}{|\Omega_m|} \sum_{
a\in \Omega_m\setminus \Ecal_m(H_0)}
\exp
\left(
s \cdot \sum_{|h|< m/2}
\gamma_m(h,t,\theta)\widehat{\varphi_a}(h)
\right)
&= \E \left[\exp \left(s \cdot \sum_{|h|< m/2} \gamma_m(h,t,\theta)\X(h)\right) \right]\\
&+ O \left(\exp \left(
-\frac{2\delta H_0^{1/2}}{5\log H_0}
\right)\right),
\end{align*}
where $\gamma_m(h,t,\theta) = \Re\left(e^{\pi i \theta } \gamma_m(h,t)\right) \in\R$ and $(\X(h))_{h\in\Z}$ denotes a family of independent and identically distributed random variables with the same distribution as $\X$ in Assumption (\hyperlink{Ass Law}{Law}).
\end{pro}

\begin{proof}
Let $\Ecal_m(H_0)$ be the set defined in \eqref{eq def Ecal_m(H_j)}, then we have the desired bound on $| \Ecal_m(H_0)| $ by \eqref{eq bound Ecal(H_0)}. Moreover, if $a \not \in \Ecal_m(H_0) $ then by \eqref{1*},
we have for $H_0$  large enough,
\[
\sup_{t,\theta\in[0,1]^2}\left|\sum_{|h|< m/2}
\gamma_m(h,t, \theta )\widehat{\varphi_a}(h)
\right| = \sup_{t\in[0,1]}
\left|
\sum_{|h|< m/2}
\gamma_m(h,t)\widehat{\varphi_a}(h)
\right|
\leq (\Mp+1)\log H_0.
\]
Let $N = \left\lfloor \frac{2\delta H_0^{1/2}}{5\log H_0}\right\rfloor$ and we assume that $|s|\leq \frac{N+1}{20(\Mp+1)\log H_0}$. Then we have
\begin{align*}\frac{1}{|\Omega_m|}
\sum_{a\in \Omega_m \setminus \Ecal_m(H_0)}
\exp
\left(
s \sum_
{|h|< m/2}
\gamma_m(h,t,\theta)\widehat{\varphi_a}(h)
\right)\\
=
\sum_{k=0}^N
\frac{s^k}{k!}
\frac{1}{|\Omega_m|}
\sum_{a\in \Omega_m \setminus \Ecal_m(H_0)}
\left(
\sum_
{|h|< m/2}
\gamma_m(h,t,\theta)\widehat{\varphi_a}(h)
\right)^k
+ R,\end{align*}
where
\begin{align*}
R &\ll \sum_{k>N} \frac{|s|^k}{k!} ((\Mp +1)\log H_0)^k \\
&\ll  \sum_{k>N}
\frac{e^k}{(N+1)^k}\left( |s|(\Mp+1) \log H_0\right)^k
\leq  \sum_{k>N}
\left(\frac{ e}{20}\right)^k
\ll e^{-N},
\end{align*}
by Stirling’s formula and our assumptions on $s$ and $N$. Furthermore, note that we have for $m$ large enough, and all $a\in E^0_m\cup E^1_m\cup E^2_m$, 
\[\sum_{|h|< m/2}
|\gamma_m(h,t,\theta)\widehat{\varphi_a}(h)| \leq  \Mp+ \sum_{1 \leq |h| < m/2}
\frac{\Mp }{2|h|}\leq 3\Mp  \log m.
\]
Since $H_0\leq (\log m)^2$, then $N\leq \frac{2\delta H_0^{1/2}}{5\log H_0} \leq \frac{\delta\log m}{5\log\log m}$ and by \eqref{eq bound E^1,2,3}, for all $k\leq N,$
\begin{align*}
&\frac{1}{|\Omega_m|}
\sum_{a\in E^0_m\cup E^1_m\cup E^2_m}\left(\sum_{|h|< m/2}
\gamma_m(h,t,\theta)\widehat{\varphi_a}(h)\right)^k\\
&\leq 
m^{-\delta/2}\exp\left( \frac{\delta\log m}{5\log\log m}\log( 3\Mp \log m)\right)\\
&\ll m^{-\frac{\delta}{5}}.
\end{align*}
Now, let $j\in [\![ 0,J ]\!] ,
$ if $ a\in E_m(H_j)\setminus \Ecal_m(H_{j+1})$ where $ \Ecal_m(H_{J+1}):= E^0_m\cup E^1_m\cup E^2_m$, it follows from \eqref{eq bound E(H j)} and \eqref{1*} that for all integers
$0 \leq  k \leq  N$, for $H_0\leq H_j$ then $N\leq \frac{2\delta H_j^{1/2}}{5\log H_j} $ and we have 
\begin{align*}
&\frac{1}{|\Omega_m|}
\sum_{a\in E_m(H_j)\setminus \Ecal_m(H_{j+1})}
\left(
\sum_
{|h|< m/2}
\gamma_m(h,t,\theta)\widehat{\varphi_a}(h)
\right)^k\\
&\leq \frac{1}{|\Omega_m|}
\sum_{a\in E_m(H_j)\setminus \Ecal_m(H_{j+1})}
\left(
(\Mp+1)\log H_{j+1}
\right)^N\\
&\ll 
\exp\left(-\frac{\delta H_j^{1/2}}{2}\right)
\exp\left(\frac{2\delta H_j^{1/2}\log[(\Mp +1)\log 2H_{j}]}{5\log H_j}\right)\\
&\ll \exp\left(-\frac{\delta H_j^{1/2}}{5}\right).
\end{align*}
Since 
\[ 
\Ecal_m(H_0) = E^0_m\cup E^1_m\cup E^2_m\cup
\bigcup_{0 \leq  j\leq J}
\left(E_m(H_j)\setminus \Ecal_m(H_{j+1})\right),
\] 
by collecting these error terms, we obtain
\[
\frac{1}{|\Omega_m|}
\sum_{a\in \Ecal_m(H_0)}
\left(
\sum_
{|h|< m/2}
\gamma_m(h,t,\theta)\widehat{\varphi_a}(h)
\right)^k
\ll \exp\left(-\frac{\delta H_0^{1/2}}{5}\right).
\]
Thus, by Lemma \ref{Pro 2.2}, we obtain
\begin{align*}
&\frac{1}{|\Omega_m|}
\sum_{a\in \Omega_m \setminus \Ecal_m(H_0)}
\left(
\sum_{|h|< m/2}
\gamma_m(h,t,\theta)\widehat{\varphi_a}(h)
\right)^k\\
&=
\frac{1}{|\Omega_m|}
\sum_{{a\in \Omega_m}}
\left(
\sum_{|h|< m/2}
\gamma_m(h,t,\theta)\widehat{\varphi_a}(h)
\right)^k
+ O \left(\exp\left(-\frac{\delta H_0^{1/2}}{5}\right)\right)\\
&= \E
\left[
\left(
\sum_{|h|< m/2}
\gamma_m(h,t,\theta)\X (h)
\right)^k\right]
+ O\left( \exp\left(-\frac{\delta H_0^{1/2}}{5}\right)+m^{-1/2}
\left[C_1 \sum_{|h|<m/2}|\gamma_m(h,t,\theta)|\right]^{k}\right).   
\end{align*}
And
\[m^{-1/2}
\left[C_1\sum_{|h|<m/2}|\gamma_m(h,t,\theta)|\right]^{k}
\leq 
m^{-1/2}
\left[C\log m\right]^{N}
\ll
\exp\left(-\frac{\delta H_0^{1/2}}{5}\right)
.\]
Since $|\gamma_m(0,t,\theta)|\leq\frac{\lfloor mt\rfloor+1}{m}\leq 2$ and
$|\gamma_m(h,t,\theta)|\leq\frac{1}{2|h|}$ for $|h|\geq1$ by \eqref{Bound gamma p},
and since $k>N\geq 4$, it follows from Lemma~\ref{lem:moments h=0} and Stirling's
formula that
\begin{align*}
    &
    \left| 
    \sum_{k>N} 
\frac{s^k}{k!}
\E
\left[\left(
\sum_{|h|< m/2}
\gamma_m(h,t,\theta)\X (h)
\right)^k\right] \right| 
\leq \sum_{k>N} \frac{e^k|s|^k}{k^k} \cdot \left( 5 \Mp\log k\right)^k\\
&\ll \sum_{k>N} \left(\frac{5\Mp e|s|\log (N+1)}{N+1}\right)^k 
\ll \sum_{k>N} \left(\frac{\Mp e}{8(\Mp+1)}\right)^k 
\ll e^{- N}.
\end{align*}
Therefore we deduce that
\begin{align*}\frac{1}{|\Omega_m|}&
\sum_{a\in \Omega_m \setminus \Ecal_m(H_0)}
\exp
\left(
s  \sum_{|h|< m/2}
\gamma_m(h,t,\theta)\widehat{\varphi_a}(h)
\right)\\
=&
\sum_{k=0}^N
\frac{s^k}{k!}
\frac{1}{|\Omega_m|}
\sum_{a\in \Omega_m \setminus \Ecal_m(H_0)}
\left(
\sum_
{|h|< m/2}
\gamma_m(h,t,\theta)\widehat{\varphi_a}(h)
\right)^k
+ O \left(e^{-N}\right)\\
=&
\sum_{k=0}^N
\frac{s^k}{k!}
\E\left[\left(
\sum_
{|h|< m/2}
\gamma_m(h,t,\theta)\X(h)
\right)^k\right]
+ O \left(e^{-N} + \exp\left(-\frac{\delta H_0^{1/2}}{5}\right)e^{|s|}\right)\\
=& \E \left[\exp \left(s \sum_{|h|< m/2} \gamma_m(h,t,\theta)\X(h)\right) \right] + O \left(e^{- N}\right) .
\end{align*}
Collecting the above estimates completes the proof.
\end{proof}

An easy calculation gives (see for example page 1501 of \cite{KS16})
\begin{equation}\label{eq unif bound gamma p =gamma}
\gamma_m(h,t) 
= \gamma (h,t) +O\left(\frac{1}{m}\right),
\end{equation}
where $\gamma_m(h,t) $ is defined in  \eqref{S_m(a,t) fourier transform} and $\gamma(h,t) $ is defined in \eqref{eq def gamma(h,t)}.
Hence, to complete the proof of Proposition \ref{pro 7.2 adapted H_0}, we need the following lemma. Thus, we obtain the Laplace transform associated with the universal probabilistic model presented in Corollary \ref{THM asymp  uniform (0,1)^2} and Corollary \ref{Cor asymp T Theta}.

\begin{lem}\label{lem estim log E exp sum gamma_p} Let $m$ be a large number. Let $t, \theta\in [0,1]$ be two real numbers. Assume that $2\leq s \leq m^{1/3}$. Then
\begin{align*}
    \sum_{|h|< m/2} \log \E \left[\exp \Big(s \cdot \gamma_m(h,t,\theta)\X(h)\Big) \right] = 
    \sum_{h\in \Z} \log \E \left[\exp \left(s \cdot  \gamma(h,t,\theta)\X(h)\right) \right] +O(1),
\end{align*} 
where $\gamma_m(h,t,\theta) = \Re\left(e^{\pi i \theta } \gamma_m(h,t)\right) $ and $\gamma(h,t,\theta) = 
 \left| \Re\left(e^{\pi i \theta } \gamma(h,t)\right) \right| $.
\end{lem}

The proof of this lemma relies on the following result from \cite{ABL21}.
\begin{lem}\cite[Lemma 5.3]{ABL21}\label{estim f} Let $f_\X$ be the function defined in \eqref{def f_X}. Then we have
\begin{equation}
  f_\X(t) \ll\left\{
   \begin{array}{l cc l}
         t^2 & \text{if}  & 0\leq t<1\\
     \log(2t) & \text{if}  & 1\leq t \\ 
   \end{array}  
   \right. .
\end{equation}
\end{lem}

\begin{proof}[Proof of Lemma \ref{lem estim log E exp sum gamma_p} ]
For all $h$ such that $|h| > s^2\geq 4$, we use \eqref{Bound gamma p}, and Lemma \ref{estim f} to see that:
\begin{align*}
    \sum_{|h|>s^2} \log \E \left[\exp \left(s \cdot  \gamma_m(h,t,\theta)\X(h)\right) \right] &\ll  \sum_{|h|>s^2} \frac{s^2}{|h|^2} \ll 1,\\
    \sum_{|h|>s^2} \log \E \left[\exp \left(s \cdot  \gamma(h,t,\theta)\X(h)\right) \right]  &\ll  \sum_{|h|>s^2} \frac{s^2}{|h|^2} \ll 1.
\end{align*}
For all \(h\) with \(|h| \leq s^2\), we use \eqref{eq unif bound gamma p =gamma}, the symmetry of the distribution of \(\X(h)\), and the bound \(|\X(h)| \leq \Mp\), to obtain
\[
\log \E \left[\exp \left(s \cdot  \gamma_m(h,t,\theta)\X(h)\right)\right] = 
\log \E \left[
\exp \left(s \cdot \gamma(h,t,\theta) \X(h) 
\right) \right] +O\left(\frac{s}{m}\right).
\]
Therefore,
\begin{align*}
    \sum_{|h|\leq s^2} \log \E \left[\exp \left(s \cdot  \gamma_m(h,t,\theta)\X(h)\right) \right]  -
    \sum_{|h|\leq s^2} \log \E \left[\exp \left(s \cdot  \gamma(h,t,\theta)\X(h)\right) \right] \ll s^2\cdot\frac{s}{m}\leq 1.
\end{align*}
Combining these estimates completes the proof.
\end{proof}

\subsection{Distribution of the maximum of partial sums}

In this section, we seek an upper bound for the following set:
\[
\Ncal_m(T,\Theta, V) := \frac{1}{|\Omega_m|}\left|\left\{
a\in \Omega_m\; \; \Big|\; \;
\underset{t\in T}{\sup_{\theta\in\Theta}}\,
\Re\left(
\frac{e^{\pi i \theta}}{\sqrt{m}}\sum_{0\leq n\leq m t} \varphi_a(n)
\right)> V
\right\}\right|
\]
where $T \subset [0,1]$ and $ \Theta \subset [0,1]$ are intervals and similarly we write $\Ncal_m^- (T,\Theta, V)$ for the proportion of $a \in \Omega_m$ such that $
\sup_{\substack{t\in T\\ \theta\in\Theta}}
\Re\left(-
\frac{e^{\pi i \theta}}{\sqrt{m}}\sum_{0\leq n\leq m t} \varphi_a(n)
\right)> V$.

\begin{pro}\label{Pro final canonic unif T,Theta} Assume Assumptions (\hyperlink{Ass Law}{Law}) and (\hyperlink{Ass Tightness}{Tightness}). Let $m$ be a large enough integer. Assume that for all $s \geq \pi,$ uniformly for intervals $T\subset[0,1]$ and $\Theta\subset[0,1]$, we have
\begin{align*}
\underset{t\in T}{\sup_{\theta\in\Theta}}
\left(
\log \E \left[\exp \Big(s \cdot \sum_{h \in \Z}
   \gamma(h,t,\theta)\X(h)\Big) \right] \right)
     &\leq  \frac{\Mp}{\pi} s \log s+B_{T,\Theta} s+ O\left(\frac{s}{\log s}\right),
\end{align*}
for some constant $B_{T,\Theta}\in [\kappa- C_3/16,\kappa]$ that depends on $T$ and $\Theta$, where $ \kappa$ is defined in \eqref{def kappa} and $C_3$ is the constant in  Theorem \ref{THM asymp t theta}.
Then, there exists $V_0>0$ such that, uniformly for all intervals $T$ and $\Theta$ and for $V$ in the range $$V_0\leq V \leq (\Mp /\pi)\left(\log_2 m-2\log_3 m\right)-V_0,$$ we have:
\begin{equation}
    \Ncal_m(T,\Theta,V)
    \leq 
\exp\left(- \frac{\Mp}{\pi e} \exp\left(\frac{\pi (V-B_{T,\Theta})}{\Mp}\right)\left[1+O\left(\frac{1}{V}\right)\right] 
\right).
\end{equation}
Furthermore for the same range $V$, we obtain the same bound for $\Ncal_m^-(T,\Theta,V)$.
\end{pro}

This proposition allows us to prove Theorem \ref{THM Ncal_m(V)} and Theorem \ref{thm structure simp}.

\begin{proof}[Proof of Theorem \ref{THM Ncal_m(V)}] The implicit upper bound is obtained by combining Proposition \ref{Pro final canonic unif T,Theta} and Corollary \ref{THM asymp  uniform (0,1)^2} for $T = [0,1],\Theta = [0,1]$, and hence $B_{[0,1],[0,1]} = \kappa.$
Then, there exists $V_0>0$ such that for all real numbers $V_0\leq  V \leq$
$\dfrac{\Mp}{\pi}\big(\log_2 m -2\log_3 m\big)-V_0$, we have
\begin{align}\label{eq THM 1.1}
    \Ncal_m(V) \leq \Ncal_m([0,1],[0,1],V)+\Ncal_m^-([0,1],[0,1],V) 
    & \leq \exp\left(- \frac{\Mp}{\pi e} \exp\left(\frac{\pi (V-\kappa)}{\Mp}\right)\left[1
+O\left(\frac{1}{V}\right)\right] 
\right),
\end{align}
where $\kappa$ is defined in \eqref{def kappa}. To conclude, the implicit lower bound corresponds to \cite[Theorem~1.6]{ABL21}.
\end{proof}

\begin{proof}[Proof of Theorem \ref{thm structure simp}] 
 Let $ 0\leq  \Delta \leq 1/2$.
By Corollary \ref{Cor asymp T Theta}, for all intervals $T,\Theta$ such that 
\begin{align*}
T\times\Theta 
& \in \{[0,1/2-\Delta]\times [0,1], [1/2+\Delta,1]\times [0,1], [0,1]\times[0,1/2-\Delta^4], [0,1]\times [1/2+\Delta^4,1]\},
\end{align*}
so that
\begin{align*}
T\times\Theta &\subset [0,1]^2\setminus[1/2-\Delta,1/2+\Delta]\times[1/2-\Delta^4,1/2+\Delta^4],
\end{align*} 
we have uniformly for $\Delta$, $B_{T,\Theta} = \kappa - C_3 \Delta^4\in [\kappa- C_3/16,\kappa]$,
\begin{align*}
\sup_{\substack{t\in T\\ \theta\in\Theta}}
\left(
\log \E \left[\exp \Big(s \cdot \sum_{h \in \Z}
   \gamma(h,t,\theta)\X(h)\Big) \right] \right)
     &\leq  \frac{\Mp}{\pi} s \log s+(\kappa - C_3 \Delta^4) s+ O\left(\frac{s}{\log s}\right),
\end{align*}
By using Proposition
\ref{Pro final canonic unif T,Theta}, there exists $V_0>0$ such that, uniformly for all intervals $T$ and $\Theta$ and for $V$ in the range $V_0\leq V \leq \frac{\Mp}{\pi}\left(\log_2 m-2\log_3 m\right)-V_0$, we have:
\begin{align*}
    \Ncal_m(T,\Theta,V)+ \Ncal_m^-(T,\Theta,V)
    &\leq 
\exp\left(- \frac{\Mp}{\pi e} \exp\left(\frac{\pi (V-B_{T,\Theta})}{\Mp}\right)\left[1+O\left(\frac{1}{V}\right)\right] 
\right)\\
&\leq \exp\left(- \frac{\Mp}{\pi e} \exp\left(\frac{\pi (V-\kappa)}{\Mp}\right)\left[\exp\left(\frac{\pi C_3 \Delta^4}{\Mp}\right)
+O\left(\frac{1}{V}\right)\right] 
\right).
\end{align*}
Now, for $V\geq 32$, we choose $\Delta = V^{-1/5}$ and hence,
\[\exp\left(\frac{\pi C_3 \Delta^4}{\Mp}\right)-1 +O\left(\frac{1}{V}\right)=\frac{\pi C_3}{\Mp}\frac{1}{V^{4/5}} +O\left(\frac{1}{V}\right)
.\]
Thus, by \eqref{eq THM 1.1}, there exists $C>0$ such that
\begin{align}\label{eq N(T,theta) C_3 D^4}
\frac{1}{\Ncal_m(V)}\frac{1}{|\Omega_m|}\left|\left\{
a\in \Omega_m \;\big|\; 
\sup_{(t,\theta) \in T\times\Theta}
\left|
\Re\left(e^{\pi i\theta}\B_m(a,t)\right) \right|
>V
\right\}\right|
\leq  
\exp\left(- C\frac{e^{\pi V/\Mp}}{V^{4/5}}\right)  .
\end{align}
Moreover, by Theorem \ref{THM Ncal_m(V)}, 
\begin{align}\label{eq N(V+C_3 D^4)}
\frac{\Ncal_m\left(V+C_3 \Delta^4 \right)}{\Ncal_m(V)}
&\leq  
\exp\left(- C\frac{e^{\pi V/\Mp}}{V^{4/5}}\right) .
\end{align}
for $V \leq
\dfrac{\Mp}{\pi}\big(\log_2 m-2\log_3 m\big)-V'_0$, where $V'_0 = V_0+ C_3/16\geq V_0 +C_3 \Delta ^4$. 

\begin{align*}
\mathcal{A}_m(V)&:= \left\{
a \in \Omega_m \, :\,  
V < \M(\varphi_a)\leq V + C_3\Delta^4, \;\;
\sup_{\substack{(t,\theta) \in 
[0,1]\times [0,1] 
\\
(t,\theta) \not\in 
\left[\frac{1}{2}-\Delta,\frac{1}{2}+\Delta\right]\times \left[\frac{1}{2}-\Delta^4,\frac{1}{2}+\Delta^4\right] }}
\left|\Re\left(e^{\pi i\theta}\B_m(a,t)\right)
\right|\leq V
\right\}\\
&\subset \left\{
a \in \Omega_m \, :\,  
V < \M(\varphi_a)\leq V + C_3 V^{-4/5}, \;\;
| t_a - 1/2|\leq \frac{1}{V^{1/5}}, \;\; | \theta_a - 1/2|\leq \frac{1}{V^{4/5}}
\right\}.
\end{align*}

Hence, by \eqref{eq N(T,theta) C_3 D^4}
and \eqref{eq N(V+C_3 D^4)}, we have
\begin{align*}
\frac{1}{\Ncal_m(V)}\frac{| \left\{
a \in \Omega_m \, :\,  
V < \M(\varphi_a)\right\} 
\setminus \mathcal{A}_m(V)| }{| \Omega_m |}
&\leq  
\exp\left(- C\frac{e^{\pi V/\Mp}}{V^{4/5}}\right) .
\end{align*}
This completes the proof.
\end{proof}

\begin{proof}[Proof of Proposition \ref{Pro final canonic unif T,Theta}] Since the random variable \(\X\) is symmetric, the proof for \(\Ncal_m^-(T,\Theta,V)\) is exactly the same as that for \(\Ncal_m(T,\Theta,V)\). Hence, we only prove the latter case.

We introduce a modified version of $\Ncal_m(T,\Theta,V)$ by removing exceptional elements $a\in\Ecal_m(H_0):$
\[
\overset{\sim}{\Ncal_m}(T,\Theta, V) := \frac{1}{|\Omega_m|}\left|\left\{
a\in \Omega_m\setminus \Ecal_m(H_0) \quad \big|\quad
\underset{t\in T}{\sup_{\theta\in\Theta}}\,
\Re\left(e^{\pi i\theta}\B_m(a,t)\right)> V
\right\}\right|.
\]
Let \(H_0\leq (\log m)^2\) be a sufficiently large integer to be chosen later in terms of $V$. By Lemma \ref{lem card set E(H)}, we obtain (uniformly for $T$ and $\Theta$) that
\begin{equation}\label{eq Ncal_m = 
sim Ncal_m}
\Ncal_m(T,\Theta, V) = 
\overset{\sim}{\Ncal_m}(T,\Theta, V)+O\left(
\exp\left(-\frac{\delta H_0^{1/2}}{2}\right)\right).
\end{equation}

Let $s\in \R^+ $. Since $\un_{X>V}\leq e^{sX-sV}$, we obtain
\[
\overset{\sim}{\Ncal_m}(T,\Theta, V)
\leq \frac{1}{|\Omega_m|}\sum_{a\in \Omega_m\setminus \Ecal_m(H_0)}
\underset{t\in T}{\sup_{\theta\in\Theta}}
\exp\left(
s \Re\left(e^{\pi i\theta}\B_m(a,t)\right) -sV
\right).
\]

Let us put $ E_T =
T\cap \Scal_{H_0^{-4/3}}$ and $E_\Theta  = \Theta\cap \Scal_{H_0^{-2/3}}.$ If $T$ is too small (that is, $T\cap \Scal_{H_0^{-4/3}}=\emptyset$) then let us take $E_T = \{t_0\}$ where $t_0 \in T$, and similarly for $E_\Theta$. In all cases,
we have $ |E_\Theta|\cdot |E_T|\leq H_0^{2}$, and by
Proposition \ref{Pro a not in Ep}, 
\begin{align*}
    &\underset{|t-t'|\leq H_0^{-4/3}}{\sup_{t,t'\in [0,1]}}
    \underset{|\theta-\theta'|\leq H_0^{-2/3}}{\sup_{\theta,\theta'\in [0,1]}} \left|\Re\left(e^{\pi i\theta}\B_m(a,t)\right)-\Re\left(e^{\pi i\theta'}\B_m(a,t')\right)\right|\\
    &\ll 
    \underset{|t-t'|\leq H_0^{-4/3}}{\sup_{t,t',\theta \in [0,1]}}\left|\B_m(a,t)-\B_m(a,t')\right|+\underset{|\theta-\theta'|\leq H_0^{-2/3}}{\sup_{t',\theta,\theta'\in [0,1]}}
    \left|\Re\left(e^{\pi i\theta}\B_m(a,t')\right)-\Re\left(e^{\pi i\theta'}\B_m(a,t')\right)\right|\\
&\ll \frac{1}{\log H_0} .
\end{align*}
Let $\hat{V} = V- \frac{c}{\log H_0 }$, for some suitably large constant $c$. Then, using this last estimate together with \eqref{eq term pair h=m/2 avec theta}, we obtain
\begin{align*}
\overset{\sim}{\Ncal_m}(T,\Theta, V)
&\leq \frac{1}{|\Omega_m|}\sum_{a\in \Omega_m\setminus \Ecal_m(H_0)}
\underset{t\in E_T}{\sup_{\theta\in E_\Theta}}
\exp\left(
s \sum_{|h| < m/2}
\gamma_m(h,t,\theta)\widehat{\varphi_a}(h) +\frac{cs}{\log H_0 } -s V
\right)\\
&
\leq 
\underset{t\in E_T}{\sum_{\theta\in E_\Theta}}
\frac{1}{|\Omega_m|}\sum_{a\in \Omega_m\setminus \Ecal_m(H_0)}
\exp\left(
s \sum_{|h| < m/2}
\gamma_m(h,t,\theta)\widehat{\varphi_a}(h) -s\hat{V}
\right)\\
&\leq H_0^2
\underset{t\in E_T}{\sup_{\theta\in E_\Theta}}
\frac{1}{|\Omega_m|}\sum_{a\in \Omega_m\setminus \Ecal_m(H_0)}
\exp\left(
s \sum_{|h| < m/2}
\gamma_m(h,t,\theta)\widehat{\varphi_a}(h) -s\hat{V}
\right),
\end{align*}
where we have used that $
| E_T| \cdot|E_\Theta | 
\leq H_0^2.$
By Proposition \ref{pro 7.2 adapted H_0} and Lemma \ref{lem estim log E exp sum gamma_p}, we have under the condition $|s| \leq \frac{\delta H_0^{1/2} }{50(\Mp+1)(\log H_0)^2}$ 
\begin{align*}
\overset{\sim}{\Ncal_m}(T,\Theta, V)
&\leq H_0^2 \cdot
\underset{t\in T}{\sup_{\theta\in\Theta}}
\E\left[\exp\left(
s \sum_{|h| < m/2}
\gamma_m(h,t,\theta)\X(h) -s\hat{V}
\right)\right] +  O \left(\exp \left(
-\frac{2\delta H_0^{1/2}}{5\log H_0}
\right)\right)
\\
&
\leq H_0^2 \cdot
\underset{t\in T}{\sup_{\theta\in\Theta}}
\left(\E\left[\exp\left(
s \sum_{h \in \Z}
\gamma(h,t,\theta)\X(h) -s\hat{V}
\right)\right]e^{O(1)} \right)+  O \left(\exp \left(
-\frac{2\delta H_0^{1/2}}{5\log H_0}
\right)\right).
\end{align*}
Furthermore, by our assumption, we obtain
\[
\overset{\sim}{\Ncal_m}(T,\Theta, V)
\leq H_0^2\cdot
\exp\left( A s\log s+Bs-s\hat{V} +O\left(\frac{s}{\log s}\right)\right) +  O \left(\exp \left(
-\frac{2\delta H_0^{1/2}}{5\log H_0}
\right)\right),
\]
where $A = \frac{\Mp}{\pi}$ and $B=B_{T,\Theta}$. 
We choose $H_0$ and $s$ so that
\[ \frac{\delta H_0^{1/2} }{50(\Mp+1)(\log H_0)^2}
 = s = \exp\left(\frac{\hat{V}-B-A}{A}\right)
\]
which optimizes the exponential term. Then
\begin{align*}
\exp\left( A s\log s+Bs-s\hat{V}\right) 
&=\exp\left(s(\hat{V}-B-A)+Bs-s\hat{V}\right)\\
&=\exp\left(-A\exp\left(\frac{\hat{V}-B-A}{A}\right)\right).
\end{align*}
This choice is admissible provided that
\[
\frac{\hat{V}}{A} \leq \log\left( \frac{\delta H_0^{1/2}}{50(\Mp+1)(\log H_0)^2}\right) + \frac{B}{A} + 1.
\]
Recall that $ \hat{V} = V- \frac{c}{\log H_0 }.$ Since $B_{T,\Theta}\in [\kappa- C_3/16,\kappa]$ and $H_0\leq (\log m)^2$, where $H_0$ is sufficiently large (see Proposition \ref{Pro a not in Ep}), there exists $V_0>0$ such that for $V$ in the range 
\[V_0 \leq V \leq A\left(\log_2 m -2\log_3 m\right)-V_0,
\]
we obtain, uniformly for $T$ and $\Theta$,
\begin{align*}
&\overset{\sim}{\Ncal_m}(T,\Theta, V) \\
&\leq H_0^2\cdot
\exp\left(-\frac{A}{e}\exp\left(\frac{V-B}{A}+O\left(\frac{1}{V}\right)\right)+ O\left(\frac{\exp(V/A)}{V}\right)\right)+   O \left(\exp \left(
-\frac{2\delta H_0^{1/2}}{5\log H_0}
\right)\right)\\ 
&= 
\exp\left(-\frac{\Mp}{\pi e}\exp\left(\frac{\pi(V-B_{T,\Theta})}{\Mp}\right)\left[1
+O\left(\frac{1}{V}\right)\right] \right) 
.
\end{align*}
This completes the proof by \eqref{eq Ncal_m = 
sim Ncal_m}.
\end{proof}

\section{Alternative sets of assumptions}

\subsection{Assumptions}

Let $\{\varphi_a\}_{a\in\Omega_m}$ be a family of $m-$periodic 
complex-valued functions. We consider the following assumptions used in \cite{ABL21}.

\textbf{Assumption 1.} Uniform boundedness.
\[\max_{a\in\Omega_m} ||\varphi_a||_\infty\ll 1.\]

\textbf{Assumption 2.} Boundedness of the Fourier transform. There exists a positive constant $\Mp$ such that for all $a\in\Omega_m$ and $h\in \Z /m\Z,$
\[\widehat{\varphi_a}(h)\in [-\Mp,\Mp].
\]

\textbf{Assumption 3.} Joint distribution of the Fourier transform. 
There exists a sequence of independent and identically distributed random variables $\{\X(h)\}_{h\in\Z^*}$ supported on $[-\Mp,\Mp]$, and absolute constants $\eta\geq 1/2$ and $C_1>1$, such that for all positive integers $k\leq \log m/\log\log m,$ and all $k-$tuples $(h_1, ...,h_k)\in\big((-m/2, m/2]\setminus \{0\}\big)^k,$ we have
\[
\frac{1}{|\Omega_m|}\sum_{a\in\Omega_m} \widehat{\varphi_a}(h_1)\cdots \widehat{\varphi_a}(h_k)
 =
 \E[\X(h_1)\cdots \X(h_k)] +O\left(\frac{C_1^k}{m^{\eta}}\right).
\]
Furthermore, if we let $\X$ be a random variable with the same distribution as the $\X(h),$ then $\X$ verifies the following conditions.

\textbf{3.a.} There exists a positive constant $A>0$ such that for all $\varepsilon >0$, we have $\P(\X>\Mp-\varepsilon)\gg \varepsilon^A$ and $\P(\X<-\Mp+\varepsilon)\gg \varepsilon^A$.

\textbf{3.b.} For all integers $l\geq 0$, we have $\E[\X^{2l+1}] =0$.

\textbf{Assumption 4.} Tightness conditions with short sums. There exist absolute constants $\alpha\geq 1$ and $0<\delta<1/2$ such that for any interval $I$ of length $|I|\leq m^{1/2+\delta}$, one has
\[
\frac{1}{|\Omega_m|}\sum_{a\in \Omega_m} \left|
\frac{1}{\sqrt{m}}\sum_{n\in I} \varphi_a(n)\right|^\alpha \ll m^{-1/2-\delta}.
\]

\begin{rem}
Since the random variable is bounded, its distribution is uniquely determined by its moments (Hausdorff moment problem). Thus, Assumption 3.b. is equivalent to saying that the distribution of \(\X\) is symmetric, that is, $\X\sim -\X$.
\end{rem}

\begin{pro}\label{Pro Assumptions implication}
Let $m$ be large, and $\Fcal = \{\varphi_a\}_{a\in\Omega_m}$ be a family of $m$-periodic complex-valued functions satisfying one of the following subsets of the above assumptions.
\begin{enumerate}[label=\Alph*]
    \item Assumption 2 and Assumption 3 with $\eta>1$.

    \item Assumptions 1, 2 and Assumption 3 with $1\geq \eta>1/2$.

    \item Assumptions 1, 2, 4 and Assumption 3 with $ \eta=1/2$.
\end{enumerate}
Then we have Assumption (\hyperlink{Ass Tightness}{Tightness}).
\end{pro}

\begin{rem} 
The only difference between Assumption (\hyperlink{Ass Law}{Law}) and Assumptions $2$ and $3$ with $\eta = 1/2$ lies in the inclusion of the index $h=0$.
In \cite{ABL21}, the authors study the distribution of $\Im (\S_m(a,1/2))$ to obtain the estimate \eqref{eq ABL Nm(1/2,V)}, but
since by \eqref{S_m(a,t) fourier transform}, $\gamma_m(0,t) = \frac{\lfloor mt \rfloor+1}{m}\in \R$, the limiting distribution of $\Im (\S_m(a,1/2))$ does not reveal the term $\widehat{\varphi_a}(0)$ and hence the inclusion of the index $h=0$ becomes optional.
However, in our case, we want to construct the probabilistic model for all $t\in [0,1]$ and thus we need to know the  distribution of $\widehat{\varphi_a}(0)$ and its correlations with other Fourier coefficients.
\end{rem}

\begin{rem} \label{rem hat varphi_a(h) = hat varphi a-h(0)}
In all the examples of Section 1.2, the parameter of the family acts on the Fourier
transform by translation. Indeed, up to the $O(1)$ elements excluded from
$\Omega_p$, we have $\widehat{\varphi_a}(h) = \widehat{\varphi_{a-h}}(0)$ or $\widehat{\varphi_{a,b}}(h) = \widehat{\varphi_{a-h,b}}(0)$
or $\widehat{\varphi_{a,b}}(h) = \widehat{\varphi_{a,b-h}}(0)$. For such families the average
$\frac{1}{|\Omega_m|}\sum_{a \in \Omega_m} \widehat{\varphi_a}(h_1) \cdots
\widehat{\varphi_a}(h_k)$ only depends on the tuple $(h_1, \dots, h_k)$ up to a
global translation, so that singling out the index $h = 0$ would be artificial:
it is more natural to require \eqref{eq moment Ass Law} for all $k$-tuples in $(-m/2, m/2]^k$, as we do
in Assumption (\hyperlink{Ass Law}{Law}), rather than for tuples avoiding $0$ as in Assumption 3.
\end{rem}

\subsection{Quadratic Gauss sums}

Let \(\chi=\left(\frac{\cdot}{p}\right)\) be the quadratic character modulo \(p\). For all \(a\in \Omega_p=\F_p^\times\), we define
\[
\varphi_a(n):=
\frac{\sqrt p}{\tau(\chi)}
\left(\frac np\right)e_p(an),
\]
where
$\tau(\chi)=\sum_{n=1}^{p-1}\left(\frac np\right)e_p(n).$
In particular, the discrete Fourier transform is exactly
\[
\widehat{\varphi_a}(h)
=
\frac{1}{\tau(\chi)}
\sum_{n\in\Z/p\Z}
\left(\frac np\right)e_p((a-h)n)
=
\left(\frac{a-h}{p}\right)\in\{-1,0,1\}.
\]

This family satisfies Assumption (\hyperlink{Ass Law}{Law}). Indeed, take \(\Mp=1\), and let \(\X\) be a Rademacher random variable taking values in \(\{\pm1\}\). 
Moreover, for all $k\in \N$, $h_1,...,h_k\in\Z/p\Z$ distinct and $\alpha_1,...,\alpha_k\in \N^*,$
\begin{align*}
\frac{1}{|\F_p^*|}
\sum_{a\in\F_p^*} \prod_{i = 1}^k \widehat{\varphi_a}(h_i)^{\alpha_i}
&=\frac{1}{p-1}\sum_{a\in\F_p^*} \prod_{i = 1}^{k} \left(\frac{a-h_i}{p}\right)^{\alpha_i}\\
&=\frac{1}{p-1}\sum_{a\in\F_p^*}  \left(\frac{ \prod_{\substack{1\le i\le k\\ \alpha_i\text{ odd}}}a-h_i}{p}\right)
\end{align*}
If at least one of the $\alpha_i$ is odd, then the polynomial $P(X)=\prod_{\substack{1\le i\le k\\ \alpha_i\text{ odd}}}(X-h_i)$
is squarefree of degree at most $k$. Hence, by Weil’s bound for multiplicative character sums,
\begin{align*}
\frac{1}{|\F_p^*|}
\sum_{a\in\F_p^*} \prod_{i = 1}^k \widehat{\varphi_a}(h_i)^{\alpha_i}
&\left\{
\begin{array}{l cc l}
         =1 - O\left(\frac{1}{p }\right)  & \text{if}  & \forall i,\,\,2|\alpha_i,\\
     = O\left( \frac{k}{p^{1/2}}\right)& \text{if}  & \exists j,\,\,\alpha_j=1\,(\mod\, 2). 
\end{array}
\right.
\end{align*}
Moreover, we have the following Burgess estimate in \cite{Bur89} for any prime \(p\) and any \(N,H,a\),
$$
\sum_{N\leq n\leq N+H} \left(\frac{n}{p}\right)e_p(an) \ll  H^{2/3} p^{1/8+\varepsilon}. $$
Thus, Assumption (\hyperlink{Ass Tightness}{Tightness}) is verified with $\alpha = 4$ and $\delta =1/24$,
\[
\max_{|I|\leq p^{1/2+1/24}}\left|
\frac{1}{\sqrt{p}}\sum_{n\in I} \left(\frac{n}{p}\right)e_p(an)
\right|^{4} 
\ll_\varepsilon 
\frac{(p^{13/24})^{8/3}\cdot p^{1/2}\cdot p^{4\varepsilon}}{p^2}
\ll
p^{-1/24}.
\]

\subsection{Preliminary}

To prove Proposition~\ref{Pro Assumptions implication}, we rely on the following lemma. Its proof requires \cite[Lemma 6.1]{ABL21} under Assumptions~2 and~3 in full generality (with arbitrary \(\eta\)). The special case \(\eta=1/2\) under Assumption~(\hyperlink{Ass Law}{Law}) was stated earlier as Lemma~\ref{lemme 6.3}.

\begin{lem}\label{lem sum courte, alpha, eta, nu} Under Assumptions $2$ and $3$, let $\eta$ be as in Assumption $3$. Let $m$ be a large number. Let $\nu \in (0,1)$. Then there exists $\alpha >0$ such that for all intervals $I$ with $|I|\leq m^\nu$, one has
\[
\frac{1}{|\Omega_m|}\sum_{a\in \Omega_m} \left|
\frac{1}{\sqrt{m}}\sum_{n\in I} \varphi_a(n)\right|^\alpha \ll_{\alpha, \eta,\nu} \frac{(\log m)^\alpha}{m^\eta}.
\]  
\end{lem}

\begin{proof}Let \(I=[mt',mt]\) with \(|t-t'|\leq \frac{|I|}{m}\leq m^{\nu-1}\). Set 
\(
\beta=\frac{2(1-\nu)}{3} \) and \( \alpha=2K\in\mathbb{N}
\)
with \(K\) large enough so that \(\beta K>\eta\). 
Using the discrete Plancherel formula \eqref{S_m(a,t) fourier transform}, and removing the term \(h=m/2\) if it exists, we obtain
\begin{align*}
    &\frac{1}{|\Omega_m|}\sum_{a\in \Omega_m} \left|
\frac{1}{\sqrt{m}}\sum_{n\in I} \varphi_a(n)\right|^\alpha \\
=& \frac{1}{|\Omega_m|}\sum_{a\in \Omega_m} \left|
\sum_
{-m/2<h\leq  m/2}
(\gamma_m(h,t)-\gamma_m(h,t'))\widehat{\varphi_a}(h)
\right|^{2K}\\
\leq & \frac{2^{2K}}{|\Omega_m|}\sum_{a\in \Omega_m} \left|
\sum_{|h|< m^\beta}
(\gamma_m(h,t)-\gamma_m(h,t'))\widehat{\varphi_a}(h)
\right|^{2K}\\
+& \frac{4^{2K}}{|\Omega_m|}\sum_{a\in \Omega_m} \left|
{\sum_{m^\beta \leq|h|< m/2}}
(\gamma_m(h,t)-\gamma_m(h,t'))\widehat{\varphi_a}(h)
\right|^{2K}+  4^{2K}\left(
\frac{\Mp}{m/2}
\right)^{2K}.
\end{align*}
By \cite[Lemma 6.1]{ABL21}, since $ \beta K>\eta$, we have:
\begin{align*}
&\frac{1}{|\Omega_m|}
\sum_{a\in \Omega_m}
\left|
\sum_{m^\beta \leq|h|< m/2}
(\gamma_m(h,t)-\gamma_m(h,t'))\widehat{\varphi_a}(h)
\right|^{2K }\\ 
&\ll
\left(
\frac{16(c_0 \Mp)^2 K}{m^\beta}
\right)^K
+\frac{(4C_1 c_0\log m)^{2K}}{m^{\eta}}
\ll \frac{(\log m)^{2K}}{m^{\eta}}.
\end{align*}
Moreover, by Lemma \ref{lem gam lip and sum H<h<2H}, we obtain
\begin{align*}
\frac{1}{|\Omega_m|}
\sum_{a\in \Omega_m}
\left| 
\sum_{|h|< m^\beta }
(\gamma_m(h,t)-\gamma_m(h,t'))\widehat{\varphi_a}(h)
\right|^{2K }
&\leq
\frac{1}{|\Omega_m|}
\sum_{a\in \Omega_m}
\left(
\sum_{|h|< m^\beta }
\frac{\pi}{2}\left(|t-t'| +\frac{1}{m} \right)\Mp
\right)^{2K }\\
&\ll 
\left(
m^\beta 
m^{\nu-1} 
\right)^{2K }
\ll m^{-\eta}.
\end{align*}
Collecting the above bounds yields the result.
\end{proof}

\subsection{Proof of Proposition \ref{Pro Assumptions implication}}
\begin{proof}[Proof of the case $C$] This follows from \cite[Lemma 3.1]{ABL21} with $L = m^{1/2+\delta/4}$.
\end{proof}

\begin{proof}[Proof of the case $B$]
Let $\delta = \frac{\eta-1/2}{2}$. By Lemma \ref{lem sum courte, alpha, eta, nu}, with $\nu = 1/2+\delta$, there exists $\alpha\geq 1$ such that for every interval $I$ of length 
$|I|\leq m^{1/2 +\delta}$, we have
\[
\frac{1}{|\Omega_m|}\sum_{a\in \Omega_m} 
\left|
\frac{1}{\sqrt{m}}\sum_{n\in I} \varphi_a(n)\right|^\alpha 
\ll m^{-1/2-\delta}.
\]
Thus, we recover Assumption $4$ and hence Case \(C\), which was treated in \cite[Lemma 3.1]{ABL21}.
\end{proof}

\begin{proof}[Proof of the case $A$]
We directly show Assumption (\hyperlink{Ass Tightness}{Tightness}). Let $ 0<\delta \leq \min(1/4, \frac{\eta-1}{3})$.
Let $ \nu = \frac{1}{2}+\delta$ and $L = \lfloor  m^\delta \rfloor$, so there exists $N>0$ independent of $m$, such that $m^\nu\leq L^{N}$ and $N\delta = \nu' <1.$ Let us choose $\alpha$ by Lemma \ref{lem sum courte, alpha, eta, nu}, such that 
for all intervals $I$ with $|I|\leq m^{\nu'}$, we have
\begin{equation}\label{eq leq m (1+2delta)}
\frac{1}{|\Omega_m|}\sum_{a\in \Omega_m} \left|
\frac{1}{\sqrt{m}}\sum_{n\in I} \varphi_a(n)\right|^\alpha \ll_{\alpha, \eta,\nu'} \frac{(\log m)^\alpha}{m^\eta} \ll \frac{1}{m^{1+2\delta}}.
\end{equation}
Thus, we would like to move the maximum $\max_{|I|\leq m^{\nu}}$ inside the sum and for this, we can lose a factor $m^{1+\delta}$. So, we define the following sets $\A(k)$ of intervals for $0\leq k\leq N-1$, of the desired size $ |\A(k)| \ll m L \leq m^{1+\delta}$ namely
\[
\A(k):= \left\{ [uL^k ,vL^k] \subset [0, m] \,\Big|\, u,v\in\mathbb{N},\ |u-v|\leq L \right\}.
\]
Moreover, any interval \(I'\) with \(|I'|\leq L^{k+1}\) can be decomposed as
$
I' = I_k \sqcup J_k \sqcup I'_k,
$
where \(J_k\in \A(k)\) and \(|I_k|, |I'_k|\leq L^k\); this follows by aligning \(I'\) with the grid \(aL^k\). Thus, by the triangle inequality, we have
\begin{align*}
\left|\frac{1}{\sqrt{m}}\sum_{n\in I'} \varphi_a(n)\right|
&\leq 
\left|
\frac{1}{\sqrt{m}}\sum_{n\in I_k} \varphi_a(n)\right|
+ \left|
\frac{1}{\sqrt{m}}\sum_{n\in I'_k} \varphi_a(n)\right|
+ \left|\frac{1}{\sqrt{m}}\sum_{n\in J_k} \varphi_a(n)\right|\\
&\leq 
2\max_{|I|\leq L^k}
\left|
\frac{1}{\sqrt{m}}\sum_{n\in I} \varphi_a(n)\right|
+ \max_{J\in \A(k)}
\left|\frac{1}{\sqrt{m}}\sum_{n\in J} \varphi_a(n)\right|.
\end{align*}
 Hence, this gives 
\begin{align*}
\max_{|I| \leq L^{k+1}}\left|\frac{1}{\sqrt{m}}\sum_{n\in I} \varphi_a(n)\right|
&\leq 2\max_{|I| \leq L^{k}}
\left|
\frac{1}{\sqrt{m}}\sum_{n\in I} \varphi_a(n)\right| +
\max_{J\in \A(k)}
\left|\frac{1}{\sqrt{m}}\sum_{n\in J} \varphi_a(n)\right|.
\end{align*}
Since $m^\nu\leq L^{N}$ and $\{J\cap\N\subset [ 0, m], \, J \in \A(0) \} = \{I\cap\N\subset [ 0, m], \, |I| \leq L \} $, we deduce by induction on \(k_0\in\{0,\dots,N\}\) that
\begin{align*}
\max_{|I| \leq m^\nu}\left|\frac{1}{\sqrt{m}}\sum_{n\in I} \varphi_a(n)\right|
&\leq 
\max_{|I| \leq L^{N}}\left|\frac{1}{\sqrt{m}}\sum_{n\in I} \varphi_a(n)\right|\\
&\leq 2^{k_0} \max_{|I| \leq L^{N-k_0}}\left|\frac{1}{\sqrt{m}}\sum_{n\in I} \varphi_a(n)\right|
+\sum_{k =N- k_0}^{N-1}
2^{N-k-1}\max_{J\in \A(k)}
\left|\frac{1}{\sqrt{m}}\sum_{n\in J} \varphi_a(n)\right|\\
&\leq 
\sum_{k = 0}^{N-1} 2^{N-k-1}
\max_{J\in  \A(k)}
\left|\frac{1}{\sqrt{m}}\sum_{n\in J} \varphi_a(n)\right|.
\end{align*}
Moreover, since $N,\delta$ depend only on $\nu = 1/2+\delta$, we have the following bound
\begin{equation*}
\max_{|I| \leq m^\nu}\left|\frac{1}{\sqrt{m}}\sum_{n\in I} \varphi_a(n)\right|
\ll_\nu
\underset{0\leq k\leq N-1}{\max_{J\in \A(k)}}
\left|\frac{1}{\sqrt{m}}\sum_{n\in J} \varphi_a(n)\right|.
\end{equation*}
Taking moments, we obtain
\[
\frac{1}{|\Omega_m|}\sum_{a\in\Omega_m}
\max_{|I|\leq m^\nu}\left|
\frac{1}{\sqrt{m}}\sum_{n\in I} \varphi_a(n)
\right|^{\alpha } 
\ll
\sum_{k = 0}^{N-1}
\sum_{J\in \A(k)}
\frac{1}{|\Omega_m|}\sum_{a\in\Omega_m}
\left|\frac{1}{\sqrt{m}}\sum_{n\in J} \varphi_a(n)\right|^\alpha.
\] 
Since \(|\A(k)| \ll m^{1+\delta}\) and by \eqref{eq leq m (1+2delta)}, we conclude that
\[
\frac{1}{|\Omega_m|}
\sum_{a\in\Omega_m}
\max_{|I|\leq m^\nu}
\left|
\frac{1}{\sqrt{m}}\sum_{n\in I} \varphi_a(n)
\right|^{\alpha } 
\ll
\sum_{k = 0}^{N-1}
\sum_{J\in \A(k)}\frac{1}{m^{1+2\delta}}
\ll m^{-\delta},
\] as desired.
\end{proof}

\textsc{
Acknowledgments} The author would like to thank his Ph.D. advisor Youness Lamzouri for many helpful discussions concerning the works
\cite{ABL21},  \cite{Lam20},
\cite{LN24} and for the
useful suggestions and comments on this paper.

\end{document}